\documentclass{amsart}
  \usepackage{amscd,amssymb,epsfig, epsf}
   \usepackage{epic,eepic}

                     \numberwithin{equation}{subsection}

                     \newtheorem{propo}{Proposition}[section]
                     
                  \newtheorem{theor}[propo]{Theorem}
                     \newtheorem{lemma}[propo]{Lemma}
                     \theoremstyle{definition}

                     \theoremstyle{remark}

             \newcommand{\Lu}{\mathcal{L}}

             \newcommand{\GK}{\mathcal{G} }
                     \newcommand{\ZZ}{\mathbb{Z}}
                     \newcommand{\RR}{\mathbb{R}}

             \newcommand{\Au}{\mathcal{A}}
               \newcommand{\BB}{\mathcal{B}}
            
 \newcommand{\Su}{\mathcal{S}}

             \newcommand{\Ker}{\operatorname{Ker}}

\newcommand{\refo}{\operatorname{sym}}
\newcommand{\cro}{\operatorname{cr}}
\newcommand{\spn}{\operatorname{spn}}

              \newcommand{\card}{\operatorname{card}}
              
                     \newcommand{\id}{\operatorname{id}}

            \newcommand{\modu}{\operatorname{mod}}
\newcommand{\Int}{\operatorname{Int}}

\begin{document}
      \title{Knotoids}
                     \author[Vladimir Turaev]{Vladimir Turaev}
                     \address{%
              Department of Mathematics, \newline
\indent  Indiana University \newline
                     \indent Bloomington IN47405 \newline
                     \indent USA \newline
\indent e-mail: vtouraev@indiana.edu} 

                     \begin{abstract}
We introduce and study knotoids. Knotoids are  represented by
diagrams in a surface which differ from the usual knot diagrams in
that the underlying curve is a segment rather than a circle. Knotoid
diagrams are considered up to Reidemeister moves applied away from
the endpoints of the underlying segment.  We show that knotoids in
$S^2$ generalize knots in $S^3$ and   study the semigroup of
knotoids. We also discuss  applications to knots  and invariants of
knotoids.

                    \end{abstract}
                     \maketitle

  \section {Introduction}

Drawing a   diagram of a   knot  may be a   complicated task,
especially when the number of crossings is big. This paper was born
from the observation that one (small) step in the process of drawing
may be skipped. It is not really necessary for the underlying curve
of the diagram to be closed, i.e., to begin and to end at the
same point. A  curve $K\subset S^2$  with over/under-crossing data
and  distinct endpoints   determines a knot in $S^3$ in a canonical
way. Indeed, let us connect the endpoints of~$K$ by an arc in  $S^2$
running under the rest of $K$. This yields a usual knot diagram
in~$S^2$. It is easy to see that  the knot in $S^3$ represented by
this diagram   does not    depend on the choice of the arc and is
entirely determined by $K$. The actual drawing of the arc in
question is
 unnecessary.   This   suggests to consider  \lq\lq open" knot diagrams   which differ from the usual ones in that the
  underlying curve is an interval rather than a circle. We call such open diagrams {\it knotoid diagrams}. They yield a new, sometimes simpler way to present knots and also lead  to an elementary but possibly useful improvement    of the standard Seifert estimate  from above for
   the  knot genus.

 The study of  knotoid  diagrams also suggests a notion of a knotoid. Knotoids are defined as equivalence classes of  knotoid diagrams modulo the  usual  Reidemeister moves applied away
    from the endpoints.      We show that knotoids in $S^2$ generalize knots in $S^3$ and
     introduce and study a semigroup of knotoids in $S^2$ containing the usual semigroup of knots as the center.
     We   also discuss an extension  of several   knot invariants to knotoids.

The concept of a  knotoid  may be viewed as a generalization of the concept of a \lq\lq long knot" 
on   $\RR^2$.  More general    \lq\lq mixtures"   formed by    closed and open knotted  curves on the plane    were  introduced  by 
S. Burckel \cite{Bu} in 2007.

 The paper is organized as follows.   In Section \ref{section0} we introduce knotoid diagrams and discuss their applications to knots. We introduce knotoids in
 Section \ref{section1} and study the semigroup of knotoids   in  Sections
 \ref{section2}--\ref{section4}. The    properties of this semigroup  are formulated in Section
 \ref{section2} and  proved in Section \ref{section4}, where we use
   the technique of  theta-curves detailed  in  Section \ref{section3}.
  Sections \ref{section77} and \ref{section67}   deal with the bracket polynomial     of knotoids.
  The last two sections are concerned with skein modules of knotoids and  with  knotoids in~$\RR^2$.

  This work   was partially supported by the NSF
  grant  DMS-0904262. The author is
  indebted to Nikolai Ivanov for helpful discussions.

  \section{Knotoid diagrams and knots}\label{section0}

   \subsection{Knotoid diagrams}\label{kd1-----}  Let $\Sigma$ be a surface.  A {\it knotoid diagram} $K$  in~$\Sigma$
  is a generic immersion of the interval $[0,1]$ in the interior of $\Sigma$ whose only singularities are transversal
   double points endowed with over/undercrossing data. The images of   $0$ and~$1$ under this
   immersion are called the    {\it  leg} and the    {\it  head} of $K$, respectively.
   These two points are   distinct from each  other and from the double points; they are    called the {\it endpoints} of~$K$. We  orient $K$ from the leg to the
   head. The double points  of $K$ are    called the {\it crossings} of
   $K$. By abuse of notation, for a knotoid diagram~$K$ in~$\Sigma$, we write $K\subset \Sigma$.
  Examples of knotoid diagrams  in $S^2$  with $\leq 2$
    crossings are shown in
   Figures~\ref{fig1} and~\ref{fig2} below.

\begin{figure}[h,t]
\includegraphics[width=12cm,height=3.3cm]{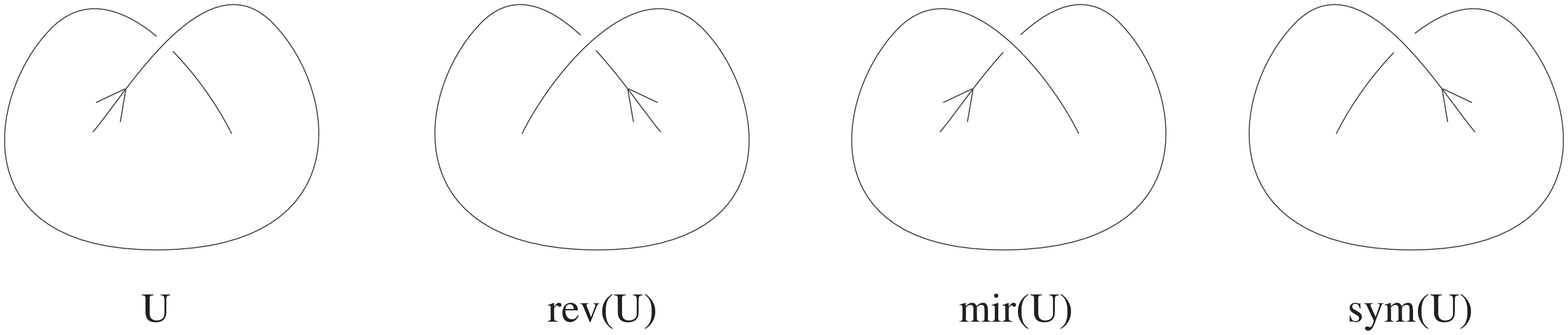}
\caption{The unifoils}\label{fig1}
\end{figure}

\begin{figure}[h,t]
\includegraphics[width=9cm,height=3.3cm]{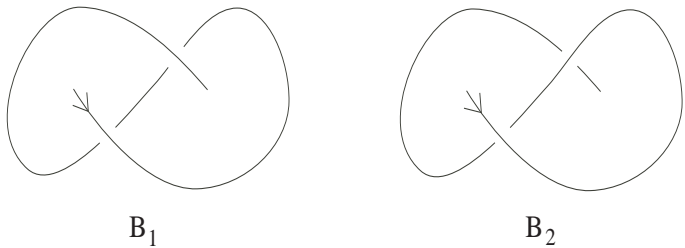}
\caption{The bifoils}\label{fig2}
\end{figure}

   Two  knotoid
   diagrams $K_1$ and $K_2$ in~$\Sigma$ are (ambient)  isotopic if there is an isotopy of $\Sigma$ in itself transforming $K_1$
     in $K_2$.
    Note that an isotopy of a knotoid
   diagram   may displace the endpoints.

  We define  three  {\it Reidemeister moves} $\Omega_1, \Omega_2, \Omega_3$ on   knotoid
   diagrams in~$\Sigma$.  The move $\Omega_i$ on a
   knotoid diagram $K\subset \Sigma$
   preserves~$K$ outside a closed 2-disk in~$\Sigma$ disjoint
   from the endpoints and modifies $K$ within this
   disk   as the standard  $i$-th Reidemeister move, for $i=1,2,3$  (pushing a branch of $K$ over/under the endpoints is not
   allowed).

   We introduce two
   more   moves on knotoid diagrams. The move $\Omega_-$ (resp.\ $\Omega_+$) pulls     the strand adjacent to the head or the leg under (resp.\ over) a transversal strand,
    see Figure~\ref{fig3}.  These moves reduce the number of crossings by $1$. Applying  $\Omega_\mp$,  we can transform any knotoid diagram
    in the   trivial  one represented by an
    embedding of $[0,1]$ in the interior of $\Sigma$.

    \begin{figure}[h,t]
\includegraphics[width=9cm,height=3.0cm]{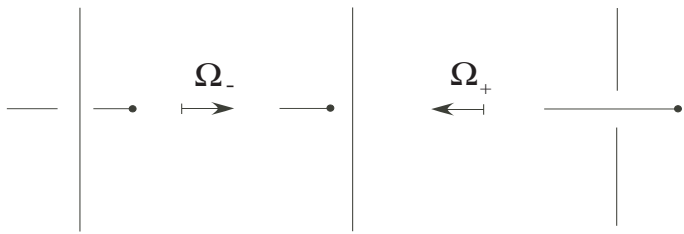}
\caption{The moves $\Omega_-$ and $\Omega_+$}\label{fig3}
\end{figure}

More general  {\it multi-knotoid diagrams} in~$\Sigma$ are defined
as
  generic immersions  of   a single oriented  segment and several
oriented circles  in~$\Sigma$ endowed with over/under-crossing data.
  Though most of the theory below
extends to multi-knotoid diagrams, we shall mainly focus on  knotoid
diagrams.

 \subsection{From  knotoid diagrams to knots}\label{section4.1}
   The theory of   knotoid diagrams suggests a new diagrammatic approach to   knots. Unless explicitly stated to the contrary, by a knot   we
   mean an isotopy class  of smooth embeddings of an oriented circle into $\RR^3$ or, equivalently,  into $S^3=\RR^3\cup \{\infty\}$.
Every knotoid diagram  $K\subset S^2$ determines a knots $K_-\subset
S^3$. It is defined as follows.  Pick an embedded arc $a\subset S^2$
connecting
  the endpoints of $K$ and otherwise meeting~$K$ transversely at a
  finite set of points distinct from the crossings of $K$. (We call such an arc  a {\it shortcut} for $K$.)   We turn   $K\cup a$ into a knot
  diagram     by declaring that $a$ passes everywhere under   $K$. The orientation of $K$ from the  leg to the head defines an
  orientation of $K\cup a$.   The  knot in $   S^3$   represented by   $K\cup a$ is denoted~$K_{-}$; we say that $K$   represents
 $K_-$ or that $K$ is a knotoid diagram of  $K_-$.
  The knot~$K_{-}$
    does  not depend on the choice of the shortcut $a$
  because any  two  shortcuts for $K$ are isotopic in  the class of embedded arcs in  $ S^2$ connecting
  the endpoints of~$K$. These isotopies
induce isotopies and Reidemeister moves on  the corresponding knot
diagrams~$K\cup a$.

  It
  is clear that every knot $\kappa\subset S^3$ may be represented by a knotoid diagram. Indeed, take a (usual) knot diagram of
  $\kappa$ and cut out  an   underpassing strand. The strand may contain
  no crossings, or 1 crossing, or $\geq 2$ crossings.
  In all cases we obtain a knotoid diagram of~$\kappa$.
It is clear that two knotoid diagrams represent isotopic knots
  if and only if these diagrams may be related by   isotopy in~$S^2$,
  the Reidemeister moves (away from the endpoints), and the moves $\Omega_-^{\pm
  1}$.

  Alternatively, one can start  with a knotoid diagram
  $K\subset S^2$ and consider the knot diagram
   obtained from $K$   by adjoining a shortcut for $K$ passing  {\it over}~$K$. This yields a knot $K_+\subset S^3$. In this context,  the
  moves $\Omega_-^{\pm 1}$ become  forbidden  and   $\Omega_+^{\pm 1}$
  allowed.

The diagrammatic approach to knots based on knotoid diagrams extends
to oriented links in $S^3$ through the use of multi-knotoid diagrams
in~$S^2$.

     \subsection{Computation of the knot group}\label{section4.2}
     Since a knotoid diagram $K\subset S^2$ fully  determines the knot    $ K_-\subset  S^3$,
     one should be able to read all invariants of this knot directly
     from $K$. We   compute here the group
     $\pi_1(S^3-K_-)$
     from $K$.

  Similarly to  the Wirtinger presentation in the theory of   knot diagrams, we associate with every knotoid diagram $K$ in an oriented surface $\Sigma$
   a {\it knotoid
  group} $\pi(K)$. This group is  defined by generators and
  relations. Observe that
   $K$ breaks at its crossings into a disjoint union of embedded \lq\lq overpassing"  segments  in $\Sigma$.
   The generators of $\pi(K)$ are associated with
   these segments. (The  generator associated with a segment is  usually represented   by a small arrow  crossing the segment
    from the right to the left.) We impose on these generators the standard Wirtinger   relations associated   with the crossings of $K$, see
     \cite{li},  p.\
   110. If $K$ has $m$ crossings, then we obtain $m+1$ generators  and $m$ relations.
  The resulting group
    $\pi(K)$ is preserved under isotopy     and   the moves $\Omega_1, \Omega_2, \Omega_3, \Omega_-$ on $K$. For example, if   $K$ is a trivial knotoid diagram, then $\pi(K)\cong \ZZ$.

     \begin{lemma}\label{le4} For any knotoid diagram $K\subset S^2$ of  a knot $\kappa \subset S^3$,
   \begin{equation}\label{groups} \pi(K)\cong \pi_1(S^3-\kappa). \end{equation}
\end{lemma}

\begin{proof}  It suffices to consider the case where $K$ has at least one crossing.
 Applying $\Omega_-^{-1}$ to $K$ several times, we can transform $K$ into a knotoid
diagram whose  endpoints lie close to each other, i.e.,  may be
connected by an arc  $a\subset S^2$ disjoint from the rest of $K$.
Then  $K\cup a$ is a knot diagram  of $\kappa=K_-$. The presentation
of $\pi(K)$   above differs from the Wirtinger presentation of
$\pi_1(S^3-\kappa)$ determined by the knot diagram $K\cup a$ in only
one aspect: the segments of~$K$ adjacent to the endpoints contribute
  different generators $g $, $h$ to the set of generators
of~$\pi(K)$. In the diagram $K\cup a$ these two segments are united
and contribute the same generator to the Wirtinger presentation.
Therefore $\pi_1(S^3-\kappa)$ is the quotient of $\pi(K)$ by the
normal subgroup generated by $g h^{-1}$. However, $g =h$ in
$\pi(K)$. Indeed, pushing a small arrow representing~$g $ across the
whole sphere $S^2$  while fixing the endpoints of the arrow and
using the relations in $\pi(K)$, we can obtain the arrow
representing $h$. Thus,   $\pi(K)\cong \pi_1(S^3-\kappa)$.
\end{proof}

A similar method allows one to associate with any knotoid diagram a
{\it knotoid quandle}, generalizing the knot quandle due to D.~Joyce
and S.~Matveev.

 \subsection{The crossing numbers}\label{section4.3}   The crossing number $\cro (\kappa)$ of a knot $\kappa \subset  S^3$ is defined as  the
 minimal number of crossings in a knot diagram of $\kappa$.
  One can use   knotoid diagrams to define two similar invariants $\cro_{\pm}(\kappa)$.  By definition,  $\cro_{\pm}(\kappa)$   is  the minimal number of crossings of a knotoid diagram $K$ such that
     $K_{\pm}=\kappa$.   Clearly,
     $\cro_{+}(\kappa)=\cro_{-}({\rm mir}
     (\kappa))$, where ${\rm mir}
     (\kappa)$ is the mirror image of $\kappa$.

     Note   that $\cro_{-}(\kappa)\leq \cro (\kappa)-1$.
This follows from the fact that a knotoid diagram of~$\kappa$ can be
obtained from a
     knot diagram of $\kappa$ with  minimal number of crossings by cutting out  an underpass
     containing one crossing. Moreover, if  a minimal diagram of  $\kappa$ has an underpass with
      $N\geq 2$ crossings, then $\cro_- (\kappa)\leq \cro (\kappa)-N$.     Similarly, $\cro_{+}(\kappa)\leq \cro
     (\kappa) -1$ and if  a minimal diagram of  $\kappa$ has an overpass with $N\geq 2$ crossings, then $\cro_+ (\kappa)\leq \cro (\kappa)-N$.



\subsection{Seifert surfaces}\label{section4.4} Recall
the     construction of a Seifert surface of a knot $\kappa$ in~$
S^3$ from a knot diagram $D$ of $\kappa$. Every crossing of $D$
admits a unique smoothing compatible with the orientation of
$\kappa$. Applying these smoothings to all crossings of~$D$, we
obtain
  a   closed oriented 1-manifold  $\widehat D\subset S^2$. This $\widehat D$ consists of
several  disjoint simple closed curves and bounds a system of
disjoint disks in $S^3$ lying   above~$S^2$. These disks together
with half-twisted strips at the crossings form a compact connected
orientable surface in $  S^3$ bounded by $\kappa$. The genus of this
surface is equal to $(\cro (D)- \vert \widehat D\vert +1)/2$, where
$\cro (D)$ is the number of crossings of~$D$ and $\vert \widehat
D\vert$ is the number of components of $\widehat D$. This yields an
estimate from above for the Seifert genus $g(\kappa)$ of $\kappa$:
\begin{equation}\label{seif} g(\kappa)\leq (\cro (D)- \vert \widehat D\vert +1)/2. \end{equation}

An analogous procedure   applies to a  knotoid diagram $K$ of
$\kappa$. Every crossing of $K$ admits a unique smoothing compatible
with the orientation of $K$ from the leg to the head. Applying these
smoothings to all crossings, we obtain an
 oriented 1-manifold  $\widehat K\subset S^2$.
This $\widehat K$ consists of an oriented interval $J\subset S^2$
(with the same endpoints as $K$) and  several  disjoint simple
closed curves. The  closed curves bound  a system of disjoint disks
in $S^3$ lying above $S^2$. We add a   band $J\times [0,1]$ lying
below $S^2$ and meeting $S^2$ along $J\times \{0\}=J$.  The union of
these disks with the band and with half-twisted strips at the
crossings is a compact connected orientable surface in $ S^3$
bounded by  $ K_-=\kappa$.    The genus of this surface is equal to
$(\cro (K)- \vert \widehat K\vert+1 )/2$, where $\cro (K)$ is the
number of crossings of $K$ and $\vert \widehat K\vert$ is the number
of components of~$\widehat K$. Therefore
\begin{equation}\label{seif+} g(\kappa)\leq (\cro (K)- \vert \widehat K\vert +1 )/2. \end{equation}
This estimate generalizes \eqref{seif} and  can be stronger. For
example, consider   the non-alternating knot $\kappa=11n1$ from
\cite{cl} represented by a knot diagram~$D$ with 11 crossings. Here
$\vert \widehat D\vert=6$ and   \eqref{seif} gives $g(\kappa)\leq
3$.
 Removing from $D$ an  underpass with 2 crossings, we obtain a
 knotoid diagram $K$ of $\kappa$ with 9 crossings and $\vert \widehat K\vert=6$.
 Formula \eqref{seif+} gives  a stronger estimate $g(\kappa)\leq 2$. (In fact, $g(\kappa)=2$, see
 \cite{cl}.)

 \section{Basics on knotoids}\label{section1}

   \subsection{Knotoids}\label{kd1} We introduce   a notion of a knotoid
  in a  surface $\Sigma$. This notion will be central in
   the rest of the paper, specifically in the case $\Sigma=S^2$.

   The    Reidemeister  moves  $\Omega_1, \Omega_2, \Omega_3$ and   isotopy   generate an equivalence relation on the set of
   knotoid
   diagrams in   $\Sigma$:  two knotoid diagrams are equivalent if they may be obtained from  each other by a finite sequence of   isotopies and the moves
   $\Omega_i^{\pm1}$ with $i=1,2,3$. The  corresponding equivalence classes are called
   {\it knotoids} in~$\Sigma$.   The set of knotoids in $\Sigma$ is
   denoted $\mathcal K (\Sigma)$.
The   knotoid    represented by an embedding $[0,1]
   \hookrightarrow \Sigma$ is said to be {\it trivial}. Any homeomorphism of surfaces $\Sigma\to \Sigma'$
   induces a bijection $\mathcal K (\Sigma)\to \mathcal K
   (\Sigma')$ in the obvious way.

   We define two commuting involutive operations on knotoids in $\Sigma$: reversion ${\rm rev}$ and  mirror
   reflection ${\rm mir}$. Reversion   exchanges the head and the  leg of a knotoid.  In other words, reversion
    inverts
   orientation on the knotoid diagrams.   Mirror reflection transforms a knotoid  into a knotoid   represented by the same diagrams
   with overpasses changed to underpasses and vice versa.

   \subsection{Knotoids in     $S^2$}\label{kd1+++} We shall be mainly interested in knotoids in   the 2-sphere $S^2=\RR^2\cup
   \{\infty\}$. They  are defined in terms of knotoid diagrams in $S^2$ as
   above. There is a convenient class of knotoid diagrams in $S^2$
   which we now define.
  A knotoid diagram $K
\subset S^2$ is   {\it normal} if the point $\infty\in S^2$ lies in
the component of $S^2-K$ adjacent to the leg of~$K$. In other words,
$K$ is normal if $K\subset \RR^2 = S^2 -\{\infty\}$ and the leg
of~$K$ may be connected to~$\infty$ by a path avoiding the rest
of~$K$. For example, the   diagrams    in Figure~\ref{fig4} below
are normal while the diagrams   in Figures~\ref{fig1} and~\ref{fig2}
are not normal.

Any knotoid $k$ in $S^2$ can be represented by a normal diagram. To
see this, take a  diagram of $k$ in $S^2$, push it away from
$\infty$ and
  push, if necessary, several branches of the diagram across
$\infty$ to ensure that the resulting diagram is normal. Note that
the  Reidemeister moves on knotoid diagrams in $\RR^2$ (away from
the endpoints)   and ambient isotopy in $\RR^2$ preserve the class
of normal diagrams. It is easy to see that two normal knotoid
diagrams represent  the same knotoid in $S^2$ if and only if they
can be related by the Reidemeister moves in $\RR^2$  and isotopy
in~$\RR^2$.

Besides reversion and mirror reflection, we consider another
involution on~$\mathcal K (S^2)$. Observe that  the reflection of
the plane $\RR^2  $ with respect to the vertical line $\{0\}\times
\RR \subset \RR^2$  extends to a self-homeomorphism of $S^2$ by
$\infty \mapsto \infty$. Applying this  homeomorphism to knotoid
diagrams in $S^2$   we obtain an involution  on $ \mathcal K (S^2)$.
This involution is called {\it symmetry} and denoted ${\rm sym}$. It
commutes with ${\rm rev}$ and ${\rm mir}$. We call these three
involutions on $\mathcal K (S^2)$   the {\it basic involutions}.

As an exercise, the reader may check that  the knotoids in $S^2$
shown in Figure~\ref{fig1} and the knotoid $B_2$ in
Figure~\ref{fig2} are trivial. The knotoids $B_1$ in
Figure~\ref{fig2} and $\varphi$ in Figure~\ref{fig4} are   equal;
 we show in the next subsection that $\varphi$ is
non-trivial. For a list of distinct knotoids represented by  diagrams with up to 5 crossings, see \cite{Ba}.

\begin{figure}[h,t]
\includegraphics[width=12cm,height=4cm]{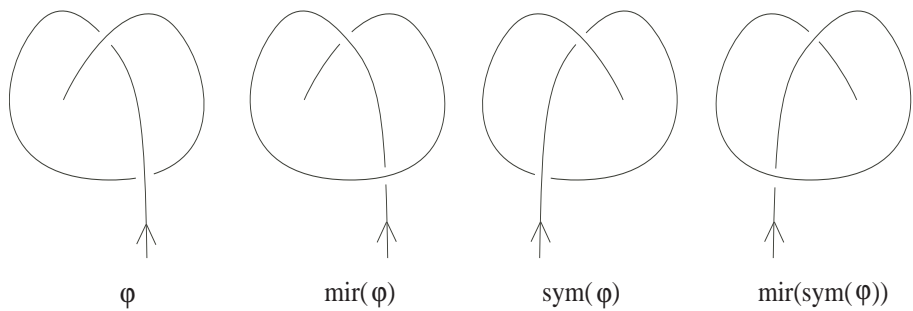}
\caption{The knotoid $\varphi$ and its transformations}\label{fig4}
\end{figure}

 \subsection{Knotoids versus knots}\label{kd2}    Every knotoid $k$ in $S^2$ determines two  knots $k_-$ and  $k_+$ in $S^3$.
  By definition $k_-=K_-$ and $k_+=K_+$, where $K\subset S^2$ is any   diagram of $k$.
  It is easy to see that   $k_{\pm}$ does  not
depend on the choice of~$K$.

In the opposite direction, every knot $\kappa \subset S^3$
determines a
 knotoid $\kappa^\bullet$ in~$S^2$. Present $\kappa$ by an oriented knot diagram $D$ in $S^2$
 and pick a small open arc $\alpha \subset D$ disjoint from the
 crossings. Then $K =D-\alpha$  is a knotoid diagram in $  S^2$ representing    $\kappa^\bullet\in \mathcal K (S^2)$.
  The diagram  $K $   may depend on the choice of~$\alpha$ but the knotoid~$\kappa^\bullet$ does not depend on this choice:
 when
 $\alpha$ is  pulled along~$D$ under (resp.\ over) a crossing   of $D$, our procedure yields an equivalent knotoid diagram.
 The
 equivalence    is achieved by pushing the    strand of~$D$
  transversal to~$\alpha$ at the crossing in question over (resp.\ under)
  $D$  towards $\infty$,
 then
 across
 $\infty$, and finally  back over (resp.\ under) $D$ from the other side of $\alpha$. (This transformation  expands as a composition of    isotopies,
  moves   $\Omega_2^{\pm 1}$, $\Omega_3^{\pm 1}$ and,
  at the very end, two moves  $\Omega_1^{- 1}$). That $\kappa^\bullet$ does not
  depend on the choice of   $D$   is clear because
  for any Reidemeister move on $D$  or a local isotopy of  $D$,
  we can  choose the arc $\alpha$ outside the disk
  where this move/isotopy modifies~$D$. To obtain a normal diagram
  of~$\kappa^\bullet$, one can apply the construction above to an
  arc $\alpha$   on an external strand of~$D$.

It is clear that  $(\kappa^\bullet)_+=(\kappa^\bullet)_-=\kappa$.
Therefore the map $\kappa\mapsto
  \kappa^\bullet$ from the set of knots   to    $\mathcal K (S^2)$ is injective.
  This  allows us to identify   knots  with the corresponding knotoids and   view $\mathcal K (S^2)$ as an extension of the set of knots.
   Accordingly, we will sometimes call the
 knotoids in $S^2$ of type $\kappa^\bullet$ {\it knots}.
  All the other knotoids in $S^2$   are said to be
   {\it pure}.  For example,   the
  knotoid $\varphi$ in $S^2$ shown in Figure~\ref{fig4} is pure because $\varphi_+ \neq
  \varphi_-$. Indeed,
  $\varphi_+$ is an unknot and  $\varphi_-$ is a  left-handed  trefoil. In particular, the knotoid $\varphi$ is non-trivial.

 The basic involutions ${\rm rev}$, ${\rm sym} $,   $ {\rm mir}$ on  $\mathcal K (S^2)$ restrict to the  orientation reversal and the reflection on knots.
Note that the restrictions of    ${\rm sym} $ and $ {\rm mir}$ to
knots are equal   because the   mirror reflections in the planes
$\RR^2 \times \{0\}$ and
  $\{0\}\times \RR^2$ are isotopic.
 The basic involutions transform pure knotoids into pure
  knotoids.

   \section{The semigroup of  knotoids}\label{section2}

\subsection{Multiplication of knotoids}\label{kd3}  Observe that each   endpoint of a knotoid diagram~$K$ in a surface~$\Sigma$ has a
  closed 2-disk
   neighborhood $B$ in $\Sigma$ such that $K$ meets $B$ precisely along a
   radius of $B$, and in particular all crossings of $K$ lie in $\Sigma - B$. We   call such     $B $ a {\it regular
   neighborhood} of the endpoint. Such neighborhoods are used in the definition of multiplication for knotoids.  Given a knotoid $k_i$ in an oriented surface $\Sigma_i$ for $i=1,2$,
 we   define a {\it product knotoid}  $k_1k_2$. Present $k_i$ by a knotoid diagram  $K_i \subset \Sigma_i$  for $i=1,2$.
Pick   regular neighborhoods $B\subset \Sigma_1$ and $B'\subset
\Sigma_2$ of the head of $K_1$ and the leg of $K_2$, respectively.
Glue $\Sigma_1- \Int (B)$ to $\Sigma_2- \Int (B')$ along a
homeomorphism $\partial B\to
\partial B'$ carrying the only point of $K_1\cap
\partial B$ to the only point of $K_2\cap \partial B'$ and such that the orientations of $\Sigma_1$,
$\Sigma_2$ extend to an orientation of the resulting surface $\Sigma
 $. The part of $K_1$ lying in
$\Sigma_1-\Int (B)$ and the part of~$K_2$ lying in $\Sigma_2-\Int
(B')$ meet   in one point and form a knotoid diagram $K_1K_2$ in~$
\Sigma $, called the {\it product} of $K_1$ and $K_2$. The knotoid
$k_1k_2 $ in~$\Sigma$ determined by $K_1K_2$ is well defined up to
orientation-preserving homeomorphisms. Clearly, if
  $\Sigma_1$, $\Sigma_2$ are connected, then
  $\Sigma=
\Sigma_1 \# \Sigma_2$.

Multiplication of knotoids is associative, and the trivial knotoid
in~$S^2$ is the neutral element. From now on, we endow    $S^2 $
with  orientation extending the counterclockwise orientation in
$\RR^2$.
 Since  $S^2\# S^2=S^2$,   multiplication of knotoids turns
 $\mathcal K(S^2)$ into a semigroup. This multiplication has a simple description in terms of normal
 diagrams: given normal diagrams $K_1$, $K_2$ of knotoids
 $k_1,k_2 \in \mathcal K(S^2)$,  we can form   $K_1K_2$   by
 attaching a copy of
 $K_2$ to the head of $K_1$ in a small neighborhood of the latter in
 $\RR^2$.
 This   implies that $$(k_1k_2)_{-}=(k_1)_{-}
 {+}  (k_2)_{-} \quad {\text {and}} \quad (k_1k_2)_{+}=(k_1)_{+}
  {+} (k_2)_{+},$$ where $ {+}$ is the
standard connected summation   of knots.
 Figure~\ref{fig5} shows  the product   of the knotoids  $\varphi,  {\rm mir}(\varphi) \in \mathcal K(S^2)$.

 \begin{figure}[h,t]
\includegraphics[width=6cm,height=5cm]{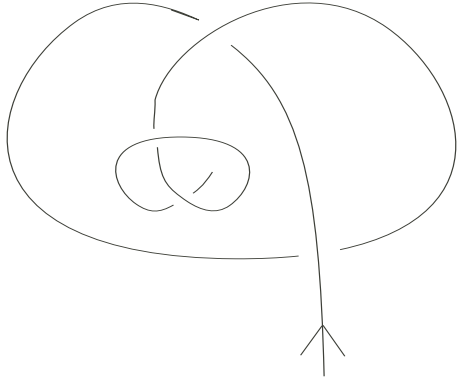}
\caption{The product $\varphi \, {\text {mir}} (\varphi)$}\label{fig5}
\end{figure}

 Given    a knotoid $k $ in $S^2$ and a knot $\kappa \subset S^3$, the
 product
  $ k  \kappa^\bullet $ is  represented by a diagram obtained  by tying  $\kappa$ in  a diagram $K $ of $k $  near the head.
 We can use the Reidemeister moves and isotopies of $\RR^2$ to pull $\kappa$ along $K $; hence, tying   $\kappa$ in any other place on
  $K $ produces the same knotoid $k  \kappa^\bullet$.
 Pulling $\kappa$ all the way through $K $ towards the leg, we obtain that \begin{equation}\label{center}
 k  \kappa^\bullet =\kappa^\bullet k . \end{equation}
 Thus,      knots lie in
 the center of the semigroup   $\mathcal K(S^2)$.

 Observe   that
  multiplication of knotoids in $S^2$ is compatible with the
  summation   of knots: $(\kappa_1  {+} \kappa_2)^\bullet
=\kappa_1^\bullet
 \, \kappa_2^\bullet
 $ for any knots $\kappa_1,
 \kappa_2\subset S^3$.

 \subsection{Prime knotoids}\label{kd4} We call a knotoid $k\in \mathcal K(S^2)$ {\it prime} if it is non-trivial, and a
 splitting $k=k_1k_2$ with $k_1,k_2\in \mathcal K(S^2)$ implies that $k_1$ or $k_2$ is the trivial
 knotoid. The next theorem says that for knots, this notion is
 equivalent to the standard notion of a prime knot.

     \begin{theor}\label{th1} A knot $\kappa \subset S^3$ is
     prime if and only if the knotoid $\kappa^\bullet\in \mathcal K(S^2)$
     is prime.
  \end{theor}

  One direction   is obvious: if $\kappa$ is as a
sum
  of non-trivial knots, then $\kappa^*$ is a product
  of non-trivial knotoids. The converse as well as the next theorem will be  proved in Section \ref{section4} using the results of Section \ref{section3}.

   \begin{theor}\label{th2} Every  knotoid in $ S^2 $ expands as a product of prime knotoids. This expansion is
   unique up to the identity  \eqref{center}, where  $k$ runs over prime knotoids and $\kappa$ runs over prime
   knots.
  \end{theor}

These theorems have interesting corollaries. First of all, the
product of two non-trivial knotoids cannot be a trivial knotoid.
Secondly,   the product of two knotoids cannot be a knot, unless
both knotoids are knots. Thirdly, every knotoid expands uniquely as
a product $\kappa^\bullet k_1k_2\cdots k_n$, where $\kappa$ is a
knot in $S^3$ (possibly, trivial), $n\geq 0$, and $k_1,k_2, \ldots,
k_n$ are pure prime knotoids in $S^2$. In more algebraic terms, we
obtain that   $\mathcal K(S^2)$ is the direct product of the
semigroup of knots and the subsemigroup of $\mathcal K(S^2)$
generated by pure prime knotoids. This subsemigroup is free on these
generators. The semigroup of knots is precisely the center of
$\mathcal K(S^2)$.


 \subsection{Complexity}\label{kd5}
    The {\it complexity} $c(K)$ of a knotoid
 diagram $K\subset S^2$ is the minimal integer~$c$ such that there
 is a shortcut $a\subset S^2$ for  $K$  whose interior meets~$K$   in $c$
  points    (the endpoints of $a$ are not counted). The {\it complexity} $c(k) $ of
  a knotoid $k\in \mathcal K(S^2)$ is the minimum of the
  complexities of the diagrams of $k$. It is clear that  $c(k)\geq 0$ and    $c(k)=0$ if and only if $k$ is a knot.
  A knotoid $k$ is  pure if and only if $c(k)\geq 1$. The complexity of a knotoid is preserved under the basic involutions.
  For example, the knotoid $\varphi$ in Figure~\ref{fig4} satisfies 
  $$c(\varphi)=c({\rm mir } (\varphi))= c({\refo } (\varphi)) =c({\rm rev }
  (\varphi))=1.$$

  Since the complexity of
  a knotoid diagram   is invariant under isotopies   in
  $S^2$, to compute the complexity of   a knotoid     we may safely restrict ourselves to
  normal diagrams and the    shortcuts in $ \RR^2$. It is easy to deduce   that
  $c(k_1 k_2)\leq c(k_1)+c(k_2)$ for any   $k_1, k_2 \in \mathcal K(S^2)$. The
  following theorem, proved in  Section \ref{section4}, shows that this inequality is in fact an
  equality.

 \begin{theor}\label{th3}  We have $c(k_1 k_2)= c(k_1)+c(k_2)$ for any  $k_1, k_2 \in \mathcal
 K(S^2)$.
  \end{theor}

  Theorem \ref{th3}  implies that   $k\mapsto c(k)$ is
  a homomorphism from the semigroup $\mathcal K(S^2)$ onto the additive semigroup of non-negative integers $\ZZ_{\geq 0}$.

 \section{A digression on theta-curves }\label{section3}

 \subsection{Theta-curves}\label{tc1} A {\it theta-curve} $\theta$
 is
 a graph embedded in $S^3$ and formed by  two  vertices $v_0, v_1$ and
  three edges $e_-, e_0, e_+$ each of which joins $v_0$ to $v_1$. We call $v_0$ and $v_1$ the {\it leg} and the {\it head} of $\theta$ respectively.
  Each   vertex $v\in \{v_0, v_1\}$ of $\theta$   has a
  closed 3-disk
   neighborhood $B \subset S^3$ meeting $\theta$ along precisely 3 radii of $B$.  We   call such     $B $ a {\it regular
   neighborhood} of $v$.
   The sets
 $$\theta_-=e_0\cup e_-,\quad \theta_0=e_-\cup e_+, \quad \theta_+=e_0\cup
 e_+$$ are    knots in $S^3$ which we   orient  from $v_0$ to $v_1$ on $e_0\subset \theta_-$, $e_-\subset \theta_0$, and $e_+ \subset \theta_+$.
 These   knots are called the {\it constituent knots} of   $\theta$.

By isotopy of theta-curves, we mean ambient isotopy in $S^3$
preserving the   labels $0$, $1$ of the vertices and the labels $-$,
$0$, $ +$ of the edges. The set of isotopy classes of theta-curves
will be denoted   $ \Theta$.

All theta-curves lying in $S^2\subset  S^3$ are isotopic to each
other. They are called {\it trivial} theta-curves. The isotopy class
of trivial theta-curves is denoted by $1$.

   Given a knot
$\kappa\subset S^3$, we can tie it   in the $0$-labeled edge of a
trivial theta-curve. This yields a theta-curve   ${\tau}(\kappa)$.
It is obvious that $({\tau}(\kappa))_0$ is a trivial knot and
\begin{equation}\label{ado} ({\tau}(\kappa))_-=({\tau}(\kappa))_+=\kappa . \end{equation}
 This implies   that   ${\tau}(\kappa) =1$   if and only if $\kappa$ is a trivial knot. Similarly, tying
$\kappa $   in the $\pm$-labeled edge of a trivial theta-curve, we
obtain  a  theta-curve  ${\tau}^{\pm} (\kappa)$.

\subsection{Vertex multiplication}\label{tc1+} The set $\Theta$ has a binary operation   called the {\it vertex
multiplication}, see \cite{wo}. It is defined as follows.   Given
theta-curves $\theta, \theta'$,
    pick   regular neighborhoods $B $ and $B' $ of the head of $\theta $ and of the leg of $\theta'$, respectively.
Let us  glue   the closed 3-balls $S^3- \Int (B)$ and    $S^3- \Int (B')$  along an orientation-reversing
homeomorphism $\partial B\to
\partial B'$ carrying the only point of $
\partial B$ lying on the $i$-th edge of $\theta $ to the only point of $  \partial B'$ lying on the $i$-th edge of $\theta'$ for $i=-, 0, +$.
(The orientation in $\partial B$, $
\partial B'$ is induced by the right-handed orientation in $S^3 $ restricted to $B, B'$.)
 The part of $\theta $ lying in
$S^3-\Int (B)$ and the part of $\theta'$ lying in $S^3-\Int (B')$
meet   in 3 points and form a  theta-curve in $  S^3$ denoted
$\theta \cdot \theta'$ or  $\theta \theta'$. This theta-curve  is
well defined up to isotopy. Observe that
$$ (\theta \theta')_-=\theta_-  {+} \theta'_-, \quad
 (\theta \theta')_0=\theta_0 {+} \theta'_0, \quad
(\theta \theta')_+=\theta_+ {+} \theta'_+.$$

It is obvious that   vertex multiplication is associative. It turns
$\Theta$ into a semigroup with neutral element $1$  represented by
the trivial  theta-curves. Note one important property of $\Theta$:
the vertex product of two theta-curves is trivial if and only if
both factors are trivial (see \cite{wo}, Theorem 4.2 or \cite{mo1},
Lemma 2.1).

It follows from the definitions that the map $\kappa\mapsto
{\tau}(\kappa)$ from the semigroup of knots to $\Theta$ is a
semigroup homomorphism: for any knots $\kappa_1, \kappa_2 \subset
S^3$,
\begin{equation}\label{homo}
{\tau}(\kappa_1+\kappa_2)={\tau}(\kappa_1) \cdot  {\tau}(\kappa_2).
\end{equation}
The image of this  homomorphism lies in the center of $\Theta$:
   pulling  a knot~$\kappa$ along the 0-labeled edge, we
easily obtain that ${\tau} (\kappa)  \cdot  \theta =\theta  \cdot
{\tau}(\kappa)$ for any theta-curve~$\theta$. Similarly,  the   maps
$\kappa\mapsto {\tau}^+ (\kappa)$ and $\kappa\mapsto {\tau}^-
(\kappa)$ are  homomorphisms from the semigroup of knots to the
center of $\Theta$.

\subsection{Prime theta-curves}\label{tc3} A theta-curve   is {\it prime} if it is non-trivial and does not split as a vertex product
of two non-trivial theta-curves. This definition is parallel to the
one of a  prime knot: a knot in $S^3$ is prime if it is non-trivial
and does not split as a connected  sum of two non-trivial knots. The
following lemma relates these two definitions.

\begin{lemma}\label{le1} A knot $\kappa\subset S^3$ is prime if and only if the theta-curve ${\tau}(\kappa)$ is prime.
\end{lemma}

We postpone the proof of Lemma \ref{le1} to the end of the section.
This lemma  implies two similar claims: a knot $\kappa $ is prime if
and only if  the theta-curve~${\tau}^+(\kappa)$ is prime;  a knot
$\kappa $ is prime if and only if  the theta-curve
${\tau}^-(\kappa)$ is prime.

\subsection{Prime decompositions}\label{tc4} Tomoe Motohashi \cite{mo1}
 established the following decomposition theorem for theta-curves: every theta-curve  $\theta$ expands as a (finite) vertex product
of prime theta-curves; these prime theta-curves  are   determined
by~$\theta$ uniquely up to permutation.  A more precise version of
the uniqueness   is given in \cite{mo2}, Theorems 1.2 and 1.3 (see also \cite{mmm}) : the
expansion $\theta=\theta_1 \theta_2 \cdots \theta_m$   as a product
of prime theta-curves is unique up to the transformation replacing
$\theta_i\theta_{i+1}$ with $ \theta_{i+1} \theta_i $ (where $i=1,
\ldots, m-1$) allowed whenever $\theta_i$ or $ \theta_{i+1} $ is the
theta-curve ${\tau}(\kappa)$, ${\tau}^{+}(\kappa)$ or
${\tau}^{-}(\kappa)$ for some knot $\kappa\subset S^3$.

\subsection{Simple theta-curves}\label{tc5} We call a theta-curve $\theta$ {\it simple} if the associated constituent  knot $\theta_0 $ is trivial. For example, the trivial theta-curve is simple. For
a non-trivial knot $\kappa\subset S^3$, the theta-curve
${\tau}(\kappa) $ is simple while ${\tau}^+(\kappa) $ and
${\tau}^-(\kappa) $ are not simple. A theta-curve isotopic to a
simple theta-curve is itself simple.

 The identity $
(\theta \theta')_0=\theta_0 {+} \theta'_0$ implies that  the vertex
product of two  theta-curves  is simple   if and only if both
factors are simple. The isotopy classes of simple theta-curves form
a sub-semigroup of   $\Theta$ denoted~$\Theta^s$. Clearly,
$\Theta^s$ is the kernel of the   homomorphism $\theta\mapsto
\theta_0$ from  $\Theta$ to the semigroup of knots in $S^3$.

The Motohashi prime decomposition theorem   specializes to simple
theta-curves as follows:
 every simple theta-curve expands as a  product $\theta_1 \theta_2 \cdots \theta_m$
of prime simple theta-curves; this expansion is unique  up to the
transformation replacing  $\theta_i\theta_{i+1}$ with $ \theta_{i+1}
\theta_i $ (where $i=1, \ldots, m-1$) allowed whenever
$\theta_i={\tau}(\kappa)$ or $ \theta_{i+1}={\tau}(\kappa) $
 for some knot
$\kappa\subset S^3$.

 \subsection{Complexity  of simple theta-curves}\label{tc5new}  We introduce   a numerical \lq\lq complexity"  of  a simple theta-curve $\theta$.
  By assumption,   there is an embedded 2-disk $D\subset S^3$ such that $\partial D=\theta_0 $ is the union of the edges
  of $\theta$ labeled by $-$ and $+$. Deforming slightly   $D$ in $S^3$ (keeping $\partial D$), we can assume that the disk interior  $\Int (D)$
   meets the $0$-labeled edge  of $\theta$   transversely at a finite number of points.
 We call such $D $
  a {\it spanning disk} for~$\theta$.
 The minimal number of the  intersections of the interior of a spanning disk for~$\theta$  with the $0$-labeled edge  of $\theta$  is called
   the {\it complexity} of $\theta $ and denoted  $c(\theta)$.  It is clear that $c(\theta) \geq 0$ is an isotopy invariant  of $\theta$ and
    $c(\theta)=0$ if and only if $\theta={\tau}(\kappa)$ for a  knot  $\kappa\subset
    S^3$.

   \begin{lemma}\label{le2} For any simple  theta-curves $\theta_1$, $\theta_2$,
   \begin{equation}\label{addit} c(\theta_1\theta_2)=c(\theta_1)+c(\theta_2). \end{equation}
\end{lemma}

\begin{proof}  Consider   the  theta graph $\theta=\theta_1\theta_2$.  The inequality $c(\theta )\leq c(\theta_1)+c(\theta_2) $  is
 obtained through gluing of spanning disks for $\theta_1$,  $\theta_2$ into a spanning disk for $\theta$. We    prove the opposite inequality.
   By the definition of the vertex multiplication, there is a 2-sphere $\Sigma \subset S^3$
  that splits $S^3$ into two 3-balls~$B_0$ and~$B_1$
     containing the leg and the head of $\theta$ respectively, such that  $\Sigma$ meets each edge of  $\theta$ transversely at one point
and $(S^3, \theta_{i})$ is obtained from  $(S^3, \theta)$ by  the
contraction
      of $B_{1-i}$ into a point   for $i=0, 1$.
Let $D\subset S^3$ be a  spanning disk  for $\theta$ whose interior
meets the $0$-labeled edge  of~$\theta$  transversely in $c(\theta)$
points. The sphere $\Sigma$ meets $\partial D$ transversely in two
points. Deforming $\Sigma$, we can additionally assume that $\Sigma$
meets $D$ transversely along a proper embedded arc and a system of
disjoint embedded  circles. Pick an innermost such circle $s\subset
\Int (D)$. The circle $s$ splits $\Sigma$ into two hemispheres
$\Sigma_0$,   $\Sigma_1$ and bounds a disk $D_s\subset D$ such that
 $\Sigma\cap D_s=\partial D_s=s$.      For $i=0, 1$, the hemisphere  $\Sigma_i$  meets $\theta$ in $n_i$ points with $0\leq n_i \leq 3$. The  union $D_s\cup \Sigma_i$ is a  2-sphere embedded  in $S^3$ and meeting $\theta$ in $n_i$ points.   Note that  the graph $\theta$ with all edges oriented from the leg to the head
is a  1-cycle in $S^3$ modulo~3. Since the algebraic intersection
number of such a cycle with $D_s\cup \Sigma_i$ is  zero,  $n_i\neq
1$. Since $n_0+n_1=\card ( \theta \cap \Sigma)=3$, one of the
numbers $n_0$, $n_1$ is equal to zero. Assume for concreteness that
$n_0=0$. Then the sphere $D_s\cup \Sigma_0$ is disjoint from
$\theta$. This sphere bounds a 3-ball in $S^3$ disjoint from
$\theta$. Pushing $\Sigma_0$   in this ball towards~$D_s$   and then
away from~$D$, we can isotope $\Sigma$ in $S^3$ into a new position
so that   $\Sigma \cap D$ has one component less. Proceeding by
induction, we reduce  ourselves to the case  where $  \Sigma  \cap
D$ is a single arc.

Under isotopy of $\Sigma$ as above, the balls $B_0$, $B_1$ bounded by $\Sigma$  follow along and keep the properties   stated at the beginning of the proof. The arc $  \Sigma  \cap D$ splits $D$ into two half-disks $D \cap B_0$ and $D\cap B_1$ pierced by the $0$-labeled edge  of~$\theta$  transversely in    $m_0$ and $m_1$ points respectively.
By the choice of $D$, we have   $m_0+m_1=c(\theta)$. On the other hand, for $i=0, 1$,  the contraction  of $B_{1-i}$ into a point transforms
 $D \cap B_i$ into a spanning disk for $\theta_{i}$
 pierced by the $0$-labeled edge  of $\theta_i$  transversely in~$m_i$ points.
 Thus, $m_i \geq c(\theta_i)$. Hence $c(\theta)=m_1+m_2 \geq c(\theta_1)+c(\theta_2)$.
\end{proof}

\subsection{Proof of Lemma \ref{le1}}  One direction is obvious: if $\kappa$ splits as a sum of two non-trivial knots $\kappa_1$, $\kappa_2$, then
 ${\tau}(\kappa)$ splits as a product of the
non-trivial theta-curves ${\tau}(\kappa_1)$ and ${\tau}(\kappa_2)$.
Suppose   that $\kappa$ is prime. We assume that ${\tau}(\kappa)$
splits as a vertex product  of two non-trivial theta-curves   and
deduce a contradiction.

Recall that ${\tau}(\kappa)$ is obtained  by tying $\kappa$ on the
0-labeled edge of a trivial theta-curve $\eta\subset S^3$. We assume
that the tying proceeds inside a small closed 3-ball $B\subset S^3$
such that  $B\cap \eta$ is a
  sub-arc of the 0-labeled edge of $\eta$.  The knot
$\kappa$ is tied in   this sub-arc  inside $B$. The resulting
(knotted) arc in $B$ is denoted by the same symbol $\kappa$, see
Figure~\ref{fig6}. In this notation, ${\tau}(\kappa)=(\eta-B)\cup
\kappa$. The arcs $B\cap \eta$ and $\kappa$ have the same endpoints
$a,b$; these endpoints are  called the {\it poles} of   $ B$.

\begin{figure}[h,t]
\includegraphics[width=9cm,height=4cm]{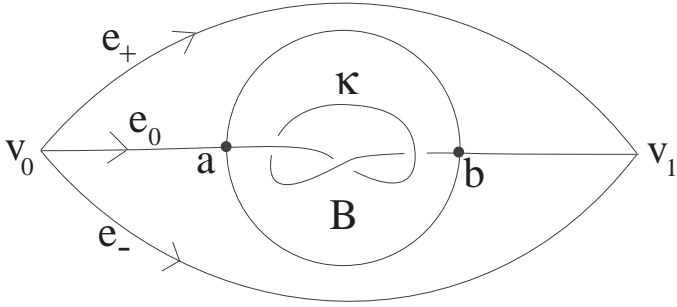}
\caption{The theta-curve ${\tau}(\kappa)$}\label{fig6}
\end{figure}

Let $\Sigma\subset S^3$ be a 2-sphere meeting   each edge of $\tau(
\kappa)$ transversely in one point and  exhibiting $\tau( \kappa)$
as a   product of   non-trivial theta-curves $\theta_1 $ and $
\theta_2$.   Slightly deforming $\Sigma$, we can assume that $a,
b\notin \Sigma$ and $\Sigma$ intersects $\partial B$ transversely.
We shall   isotope $\Sigma$ in $S^3 $ (keeping the requirements on
$\Sigma$ stated above) in order to reduce $\Sigma \cap \partial B$
and eventually to obtain $\Sigma\cap
\partial B=\emptyset$. Then    $\Sigma\cap B=\emptyset$ and~$\Sigma$ exhibits
$\eta$ as a product   of two theta-curves $\theta'_1$ and $
\theta'_2$. One of them is disjoint from~$B$ and coincides with
$\theta_1$ or $\theta_2$. By the Wolcott theorem stated in
Section~\ref{tc1+}, $\theta'_1= \theta'_2=1$. This contradicts the
non-triviality of $\theta_1$ and $\theta_2$.

The components of $\Sigma \cap \partial B$ are  circles   in the
2-sphere $\partial B$  disjoint from each other and from the  poles
$ a,b \in \partial B$. Suppose that  one of these circles, $s$,
bounds a disk $D_s$ in $ \partial B- \{a,b\}$.  Replacing if
necessary $s$ by an innermost component of $\Sigma \cap \partial B$
 lying in this disk, we can assume that   $
\Sigma \cap \Int (D_s)=\emptyset$. The circle~$s$ splits $\Sigma$
into two hemispheres. The same argument as in the proof of
Lemma~\ref{le2} shows that one of these hemispheres is disjoint from
${\tau}(\kappa)$ and its union with $D_s$ bounds a ball in
$S^3-{\tau}(\kappa)$. Pushing this hemisphere   inside this ball
towards $D_s$   and then away from $\partial B$, we can isotope
$\Sigma$ in $S^3$ into a new position so that   $\Sigma\cap \partial
B$ has one component less. Proceeding by induction, we reduce
ourselves to the case  where all components of $\Sigma \cap \partial
B$ separate the poles $a$, $b$ in $\partial B$. In particular, the
linking numbers of any component of $\Sigma \cap \partial B$ with
the   constituent knots $({\tau}(\kappa))_{+}$ and
$({\tau}(\kappa))_{-}$ are equal to $\pm 1$.

If $\Sigma\cap B$ has a disk component, then this disk   meets
$\kappa\subset B$ in one point and splits $B$ into two balls $B_1$,
$B_2$.   Since $\kappa$ is prime, one of the ball-arc pairs $(B_1,
B_1\cap \kappa)$, $(B_2, B_2\cap \kappa)$ is trivial. Pushing
$\Sigma$  away across this ball-pair, we can isotope~$\Sigma$
in~$S^3$ into a new position such that $\Sigma\cap
\partial B$ has one component less (and still all these components
separate the poles). Thus, we may assume that $\Sigma\cap B$ has no
disk components.

Let $B^c=S^3-\Int (B)$ be the complementary 3-ball of $B$. If
$\Sigma\cap B^c$ has a disk component $D$, then the linking number
considerations show that  either $D$ meets  the 0-labeled edge  of
${\tau}(\kappa)$  in at least one  point or $D$ meets  each of the
other two edges  of ${\tau}(\kappa)$  in at least one point.
Since~$\Sigma$ meets each edge of  ${\tau}(\kappa)$ only in one
point, $\Sigma\cap B^c$ may have at most two disk components. This
implies that the 1-manifold $\Sigma\cap \partial B$ splits~$\Sigma$
into  several annuli and two disks lying in $B^c$. One of these two
disks, say $D_1$, meets the 0-labeled edge  of ${\tau}(\kappa)$
 in  one  point $d$.
Observe that the intersection of the 0-labeled edge  of
${\tau}(\kappa)$ with $B^c$ has two components containing the poles
$a, b\in \partial B=\partial B ^c$. Assume, for concreteness,
 that   $d$ and   $a $ lie in the same component of this intersection. The circle $\partial D_1\subset \Sigma \cap \partial B$ bounds a disk
 $D_a$ in $\partial B$ containing $a$ (and possibly containing other components of $\Sigma\cap \partial B$).
 The union $D_1\cup D_a\subset B^c$ is an embedded 2-sphere meeting $\eta$ in
$a$ and $d$. This sphere bounds a 3-ball $B_+\subset B^c$ whose
intersection with $\eta$ is the sub-arc of the 0-labeled edge
of~$\eta$ connecting $d$ and $a$. The triviality of $\eta$ implies
that this arc is   unknotted in $B_+$.  Pushing $D_1$  in  $B_+$
towards  $D_a$ and then inside $B$, we can  isotope $\Sigma$
in~$S^3$ into a new position so that   $\Sigma\cap \partial B$ has
at least one component less. This isotopy creates  a disk component
of $\Sigma\cap  B$  which can be further eliminated as explained
above. Proceeding recursively, we eventually  isotope $\Sigma$ so
that it does no meet~$
\partial B$.
\qed

 \section{Proof of Theorems \ref{th1}--\ref{th3} }\label{section4}

We begin with a  geometric
 lemma.

 \begin{lemma}\label{le3} An orientation preserving diffeomorphism
 $f:S^3\to S^3$ fixing pointwise an unknotted circle $S\subset S^3$  is
 isotopic to the identity $\id:S^3\to S^3$ in the class of diffeomorphisms $S^3\to S^3$ fixing   $S $
 pointwise.
\end{lemma}

\begin{proof} Pick a
  tubular neighborhood $N \subset S^3$ of $S$. We have $N=S \times {D}$, where ${D}$ is a 2-disk and the identification $N=S \times {D}$ is chosen so
  that for $p\in \partial {D}$,
  the longitude $S\times \{p\}\subset \partial N $  bounds a disk  $D'$  in   $N^c=S^3-\Int (N)$. We can deform~$f$ in the class of diffeomorphisms of $S^3$ fixing $S$ pointwise
  so that $f(N)=N$ and~$f$ commutes with the projection $N\to S$. Then the  diffeomorphism $f\vert_{\partial N}: {\partial N}\to {\partial N}$ induces   a loop  $\alpha_f:S\to {\rm
  {Diff}}
   ({\partial D})$  in  the group   of   orientation preserving diffeomorphisms of
   the circle $ \partial D$. This group is a homotopy circle and    $\pi_1({\rm {Diff}} ({\partial D}))=\ZZ$. The
   integer corresponding to $\alpha_f$ is nothing but the linking
   number of $S$ with   $f(S\times \{p\})$. Since
  $f(S\times \{p\})=\partial f(D') $,   this linking number   is equal to 0. Thus, the
   loop $\alpha_f$ is contractible. This allows us to deform $f$   in the class of diffeomorphisms $S^3\to S^3$ fixing   $S
   $ pointwise in a diffeomorphism, again denoted $f$, such that   $f(N)=N$, $f$ commutes with the projection $N\to S$,
   and $f=\id$  on $\partial N$. Now, the  diffeomorphism $f\vert_N:N\to N$ induces   a loop   in  the group   of   orientation preserving diffeomorphisms of
   ${D}$ fixing pointwise $\partial D$ and the center of ${D}$. This group is   contractible and therefore  the
   loop in question also is contractible. This allows us to deform $f$   in the class of diffeomorphisms $S^3\to S^3$ fixing   $S
   \cup \partial N $ pointwise in a diffeomorphism, again denoted $f$, such that    $f=\id$ on $N$.
 The restriction of $f$ to the solid torus  $T=N^c$ is    an
  orientation-preserving
 diffeomorphism   fixing~$\partial T$ pointwise. Then $f\vert_T$ is
  isotopic to the identity $\id_T:T\to T$ in the class of diffeomorphisms $T\to T$   fixing~$\partial T$
  pointwise, see
  Ivanov \cite {iv}, Section 10 (the  proof of this fact uses the famous
theorem of Cerf $\Gamma_4=0$ and the work of Laudenbach \cite{la}).
  Extending the isotopy between $f\vert_T$ and  $\id_T $
  by the identity on $N$ we obtain an isotopy of $f$ to the identity
  constant on $S$.
\end{proof}

\subsection{A map $t:\mathcal
 K(S^2) \to \Theta^s  $}\label{pp1} Starting from a knotoid diagram
 $K\subset \RR^2$, we construct  a simple theta-curve   as follows. Let $v_0, v_1  $ be the leg and the head of~$K$. Pick  an  embedded arc  $a\subset \RR^2$ connecting
  $v_0$ to $v_1$. We   identify $\RR^2$ with the coordinate
  plane $\RR^2\times \{0\} \subset \RR^3 $. Let $e_{+}  $ (respectively, $e_{-}$) be the
  arc in~$\RR^3$ obtained by pushing the interior of $a$ in the
  vertical direction in the
  upper (respectively lower)
  half-space  keeping  the endpoints  $v_0, v_1 \in \RR^2\times \{0\}$.
    Pushing the underpasses of~$K$ in the lower half-space we
  transform $K$ into an embedded arc   $e_0\subset \RR^3$ that meets $e_-\cup
  e_+$
  solely at   $v_0$ and $v_1$. Then $\theta = e_-\cup
  e_0\cup
  e_+$ is a  theta-curve in $S^3=\RR^3\cup \{\infty\}$. It  is
  simple because   $e_-\cup
  e_+  =\partial D_a$   for a (unique) embedded 2-disk  $D_a \subset a\times \RR  $  such that $  D_a\cap ( \RR^2\times \{0\})=a$. The same arguments as in Section \ref{kd2} show that the isotopy class of
  $\theta$ does    not depend on the choice of~$a$ and  depends only on the knotoid   $k\in \mathcal K(S^2)$ represented by
     $K$. We denote this isotopy  class by $t(k)$. This construction defines a map  $t:\mathcal
 K(S^2) \to \Theta^s  $. For example,  if
 $k=\kappa^\bullet$ for a knot $\kappa\subset S^3$, then $t(k)=
 {\tau}(\kappa)$  is the theta-curve introduced in   Section
 \ref{tc1}.

The following theorem yields a geometric
 interpretation of knotoids in $S^2$ and computes the semigroup $\mathcal
 K(S^2) $ in terms of theta-curves.

 \begin{theor}\label{th777} The map $t:\mathcal
 K(S^2) \to \Theta^s  $ is a semigroup isomorphism.
\end{theor}

\begin{proof}  That $t$ transforms multiplication of knotoids into vertex multiplication of theta-curves follows   from the definitions. To prove that $t$ is bijective we construct the inverse map
$ \Theta^s\to \mathcal
 K(S^2) $.

 Let $\theta\subset S^3=\RR^3\cup \{\infty\}$ be  a   theta-curve  with vertices $v_0, v_1$
 and edges $e_-, e_0, e_+$. We say that $\theta$ is {\it standard}
 if $\theta\subset \RR^3$, both vertices of $\theta$ lie in $  \RR^2=\RR^2\times
 \{0\}$,   the edge  $e_+ $ lies in the upper half-space, the edge  $e_- $ lies in the lower
 half-space, and   $e_+ $,  $e_- $ project bijectively to the same embedded
 arc $a\subset \RR^2$ connecting $v_0$ and $v_1$. A
 standard theta-curve   is simple and has a  \lq\lq standard" spanning disk   bounded by
   $e_+\cup e_-$ in   $a\times \RR$.

   Observe that any simple theta
   curve~$\theta\subset S^3$ is (ambient) isotopic to a standard theta-curve. To see
 this,   isotope~$\theta$ away from     $\infty\in S^3$ so that  $\theta\subset
 \RR^3$,   pick a spanning disk for   $\theta$ and apply an
 (ambient)  isotopy pulling this disk in a standard position as
 above.

 We claim that   if two   standard theta-curves $\theta,
\theta'\subset \RR^3$ are
 isotopic, then they are isotopic in the class of standard theta-curves.  Indeed, we
 can easily deform $\theta'$ in the class of standard theta-curves so that $\theta$ and $\theta'$ share the same vertices and the same $\pm$-labeled
 edges. Let $S$ be the union of these vertices and   edges. The set $S$ is an unknotted circle in $S^3$.
 Since
$\theta$ is isotopic to~$\theta'$,   there is an
 orientation-preserving  diffeomorphism $f:S^3\to S^3$ carrying
 $\theta$ onto~$\theta'$ and preserving the labels of the vertices
 and the edges. Then $f(S)=S$. Deforming~$f$, we can additionally
  assume  that $f\vert_S=\id$. By  the previous lemma,~$f$ is isotopic to the
  identity $\id:S^3\to S^3$ in the class of diffeomorphisms fixing~$S $
 pointwise. This isotopy induces an isotopy of $\theta'$  to
  $\theta$ in the class of standard theta-curves.

 The results above show that without loss of generality we can focus on the class of standard theta-curves and their isotopies in this class. Consider a standard
theta-curve $\theta\subset  \RR^3 $.  We shall apply to $\theta$ a
sequence of  (ambient) isotopies
   moving only  the interior of  the 0-labeled edge~$e_0$ and keeping fixed
the other two edges $e_-$, $e_+$ and  the vertices. Let $a\subset
\RR^2$ be the common projection of $e_-$, $e_+$ to $\RR^2$ and let
$D\subset a\times \RR$ be the standard spanning disk for $\theta$.
First, we isotope~$e_0$ so that it meets $a\times \RR$ transversely
in a finite number of points. The intersections of~$e_0$  with
$(a\times \RR)-D$ can be eliminated by pulling the corresponding
branches of~$e_0$ in the horizontal direction across $v_0\times \RR$
or $v_1\times \RR$. In this way, we can isotope~$e_0$ so that all
its intersections with $a\times \RR$ lie inside $D$. Then we further
isotope $e_0$   so that its projection   to $\RR^2$   has only
double transversal crossings.  This projection is provided with
over/under-crossings data in the usual way and becomes
   a knotoid diagram. The knotoid
$u(\theta)\in \mathcal K (S^2)$ represented by this diagram depends
only on $\theta$ and does not depend on the choices made in the
construction. The key point is that pulling a branch of $e_0$ across
$v_0\times \RR$ or across $v_1\times \RR$ leads to equivalent
knotoids in $S^2$, cf.\ the argument in Section \ref{kd2}. All other
isotopies of~$e_0$ are translated to  sequences of   isotopies and
Reidemeister moves on the knotoid diagram away from the vertices.
Observe finally that the knotoid $u(\theta)$ is preserved under
isotopy of $\theta$ in the class of standard theta-curves. Therefore
$u$ is a
 well-defined
 map $  \Theta^s\to \mathcal
 K(S^2) $. It is clear that the maps~$t$ and~$u$ are mutually inverse.
\end{proof}

 \subsection{Proof of Theorem \ref{th1} }\label{section3.2} By Lemma \ref{le1}, a knot $\kappa\subset S^3$ is prime if and only if the theta-curve ${\tau}(\kappa)$ is prime.  As we know, $\tau (\kappa)=t(\kappa^\bullet)$.
 Theorem \ref{th777} shows that $t(\kappa^\bullet)$ is prime if and only if the knotoid $\kappa^\bullet$ is prime.

  \subsection{Proof of Theorem  \ref{th2}}\label{section3.3}  Theorem  \ref{th2}   follows from  Theorem \ref{th777}
   and the Motohashi prime decomposition theorem for simple theta-curves.

 \subsection{Proof of Theorem  \ref{th3}}\label{section3.4} We claim that for any knotoid $k$ in $S^2$, its complexity $c(k)$ is equal
  to the complexity $c(t(k))$  of the simple theta-curve $t(k)$. Observe that to compute $c(k)$ we may use only knotoid diagrams   and
  their shortcuts lying in $\RR^2$.
 For any knotoid diagram $K\subset \RR^2$ of $k$ and any shortcut $a\subset \RR^2$ for $K$,
 the number of intersection points of the $0$-labeled edge of $t(k)$ with the spanning disk $D_a$ of $t(k)$ is equal to the number
  of intersections of the interior of~$a$ with $K$. Hence $c(k)\geq c(t(k))$.  Conversely, given a spanning disk $D$ of $t(k)$ meeting
  the $0$-labeled edge transversely in
 $c(t(k))$ points, we can isotope $t(k)$ and~$D$ as in the proof of Theorem \ref{th777} so that $D\cap \RR^2$ becomes  a shortcut
 for a diagram of~$k$ in~$\RR^2$. Therefore $c(k)\leq c(t(k))$. This proves the equality $c(k)= c(t(k))$.  This equality shows that the complexity map $c: {\mathcal K (S^2)}\to \ZZ$ is the composition of   $t:{\mathcal K (S^2)}\to \Theta^s$ with  the complexity map $c:\Theta^s\to \ZZ$.  Since the latter map is a semigroup homomorphism (Lemma \ref{le2}) and so is $t$, their composition is a semigroup homomorphism.

  \subsection{Remarks}\label{section3.5} 1. The existence of a prime decomposition of any knotoid $k\in   {\mathcal K (S^2)}$
  may be proved directly without referring to the Motohashi theorem. In fact, we can prove the following stronger claim.
  Let $N\geq 0$ be the number of factors (counted with multiplicity) in the decomposition of the knot  $k_-$ as a sum of prime knots.
  Set $M=c(k)+N$. We claim that $k$ splits as a product of at most $M$ prime knotoids.
  Indeed, let us split $k$ as a product of two non-trivial knotoids and then inductively   split all non-prime factors as long as it is possible.
   This process must stop at  (at most) $M$ factors. Indeed, suppose that $k=k_1k_2\cdots k_m$ is a decomposition of $k$ as a product of $m>M$
   non-trivial
  knotoids.   Theorem  \ref{th3} gives
  $\sum_{i}  c(k_i)=c(k)$. Therefore at most $c(k)$ knotoids among $k_1, \ldots, k_m$ have positive complexity. Since $m>M=c(k)+N$, at least $N+1$
   knotoids among $k_1, \ldots, k_m$ have complexity~$0$.
 A non-trivial   knotoid $k_i$ of complexity  $0$ has  the form $\kappa^\bullet$ for a non-trivial knot $\kappa\subset S^3$.
 The knot $\kappa$ may be recovered from $k_i$ via  $\kappa=(k_i)_-$.
 We conclude that in the expansion $k_-=(k_1)_-+ (k_2)_-+ \cdots + (k_m)_-$ the right-hand side has  at least $N+1$ non-trivial summands.
 This contradicts the  choice of $N$.

 2. Given a knotoid $k$ in $S^2$, we can use the theta-curve $t(k)$   to derive from $k$ one more
 knot in $S^3$.
 Consider the 2-fold covering
 $p:S^3\to S^3$ branched along the  trivial  knot formed by the $\pm$-labeled edges of $ t(k)
 $. The  preimage under $p$ of the 0-labeled edge of $t(k)$ is a
 knot in $S^3$ depending solely on $k$.

3.  Recall the  multi-knotoid diagrams in a surface $\Sigma$
introduced in Section~\ref{kd1-----}.      The classes of such
diagrams under the equivalence
 relation generated by isotopy in~$\Sigma$ and the three Reidemeister moves (away
 from the endpoints of the   segment component) are called {\it   multi-knotoids}  in
 $\Sigma$. The definitions and the theorems of Section~\ref{section2} directly  extend to multi-knotoids in~$S^2$.
 The proofs use  the     theta-links defined as embedded  finite graphs  in~$S^3$ whose components are oriented circles except one component which is a
  theta-curve. A theta-link is simple if its theta-curve component
  is simple.
  Theorem \ref{th777} extends to this setting and establishes an isomorphism between the semigroup of multi-knotoids in $S^2$
  and the semigroup of (isotopy classes of) simple theta-links. Note
  also that the Motohashi theorems extend to theta-links, see
  \cite{mmm}.

4. The theory of knotoids offers a diagrammatic calculus for simple
theta-curves. A similar calculus for arbitrary theta-curves can be
formulated in terms of bipointed  knot diagrams. An (oriented) knot
diagram is {\it bipointed} if it is endowed with an ordered pair of
generic points, called the leg and the head. A bipointed  knot
diagram~$D$ in $S^2$ determines a theta-curve $\theta_D\subset S^3$
by adjoining an embedded arc connecting the leg to the head and
running under $D$. This arc is the 0-labeled edge of  $\theta_D$,
  the   segment of $D$ leading from the leg to the head is the
$+$-labeled edge, and the third edge   is labeled by $-$. Clearly,
any theta-curve is isotopic to~$\theta_D$ for some   $D$. The
isotopy class of~$\theta_D$ is preserved   under the Reidemeister
moves on $D$ away from the leg and the head and under
  pushing a branch of $D$ over the leg or the head. (Pushing
a branch   under the leg or the head is forbidden). These moves
generate the isotopy relation on   theta-curves.

 \section{The bracket polynomial and the crossing number}\label{section77}

 \subsection{The bracket polynomial}\label{section5.1}  In analogy with Kauffman's bracket polynomial of knots, we
  define the     bracket polynomial for   knotoids in any oriented surface~$\Sigma$.
 By a {\it state} on a knotoid diagram $K\subset \Sigma$, we mean a mapping from the set of crossings of $K$ to the   set
 $\{-1,+1\}$.
 Given a state~$s$ on $K$, we   apply the A-smoothings (resp.\ the B-smoothings)
   at all  crossings of $K$ with positive (resp.\ negative) value  of $s$.
This   yields    a compact 1-manifold $K_s\subset \Sigma$ consisting
of a single embedded segment and several  disjoint  embedded
circles. Set
 $$\langle K\rangle =\sum_{s\in S(K)} A^{\sigma_s } (-A^2-A^{-2})^{\vert s\vert
 -1} ,$$
 where $S(K)$ is the set of all states of $K$,  $\sigma_s\in \ZZ$ is the sum of the values $\pm 1$ of $s\in S(K)$ over all  crossings of
 $K$,
 and $\vert s\vert$ is the number of components of~$K_s$.
 Standard computations show that the Laurent polynomial $\langle K\rangle\in \ZZ[  A^{\pm 1}] $ is invariant
 under the second and third Reidemeister moves on~$K$ and is multiplied by $(-A^3)^{\pm 1}$ under the first Reidemeister moves.
 The polynomial $\langle K\rangle$ considered up to multiplication
 by integral powers of $-A^3$ is an invariant of  knotoids     denoted $\langle\,   \rangle$ and called the {\it   bracket polynomial}.

One   useful invariant of knotoids derived from the bracket
polynomial is the   span. The {\it span}  of a non-zero Laurent
polynomial $f=\sum_i f_i A^i\in \ZZ[A^{\pm 1}]$ is defined by $\spn
(f)=i_+-i_-$, where $i_+$ (resp.\ $i_-$) is the maximal (resp.\ the
minimal) integer~$i$ such that $f_i\neq 0$. For $f=0$, set $\spn
(f)=-\infty$. The {\it span} $ \spn (K)$ of  a knotoid diagram $K$
is defined by $\spn (K)= \spn (\langle K \rangle)$. Clearly, $ \spn
(K)$ is invariant under all Reidemeister moves on $K$ and defines
thus a knotoid invariant also  denoted $\spn$. The span of any
knotoid is an even (non-negative) integer.

The indeterminacy   associated with
 the first Reidemeister moves can be handled  using the writhe. The
 writhe  $w(K)\in \ZZ$ of a knotoid diagram~$K$ is the sum of  the signs of the crossings of~$K$   (recall that $K$ is  oriented from
 the leg to the head).
 The  product  $ \langle K\rangle_\circ= (-A^3)^{- w(K)} \langle K\rangle$  is invariant under
 all Reidemeister moves on~$K$. The resulting invariant of  knotoids  is  called the {\it normalized bracket polynomial} and denoted
 $ \langle \,\rangle_\circ$.
It  is invariant
 under the reversion of knotoids and   changes
 via
  $A\mapsto A^{-1}$ under mirror reflection and under orientation reversion    in~$\Sigma$.  The normalized bracket polynomial is multiplicative:
  given a knotoid~$k_i$ in an oriented surface
 $\Sigma_i$ for $i=1,2$,
 we have $ \langle k_1k_2 \rangle_\circ=\langle k_1  \rangle_\circ \, \langle  k_2 \rangle_\circ  $.
This implies that   the span of knotoids is additive with respect to
multiplication of knotoids.

\subsection{An estimate of the crossing number}\label{section5.2} A   fundamental property  of the bracket polynomial
of knots established by L.\ Kauffman \cite{ka} is an inequality
relating the span  to the crossing number. This   generalizes to
knotoids as follows.

 \begin{theor}\label{thhspann}  Let   $\Sigma$ be an oriented
 surface. For any knotoid diagram $K \subset \Sigma$ with $n$ crossings,   \begin{equation}\label{cc-}\spn (K)\leq
 4 n .\end{equation}
  \end{theor}

\begin{proof}   Let ${s_+}$
(resp.\ ${{s_-}}$) be the state of $K$ assigning $+1$ (resp.\ $-1$)
to all crossings. The same argument as in the case of knots shows
that
\begin{equation}\label{cc} \spn (K)=\spn (\langle
K \rangle)\leq 2(n +\vert {s_+}\vert +\vert {{s_-}}\vert-2).
\end{equation} To
estimate $\vert {s_+}\vert +\vert {{s_-}}\vert$,  we need the
following construction   introduced for knot diagrams in \cite{tu1}.
Let $\Gamma\subset \Sigma$ be the underlying graph of $K$. This
graph is connected and has  $n$ four-valent vertices, two 1-valent
vertices (the endpoints of~$K$), and $2n+1$ edges. We thicken
$\Gamma$ to a surface: every vertex is thickened to a small square
in $\Sigma$ and every edge $e$ of $\Gamma$   is thickened to a band.
If one endpoint of $e$ is 1-valent or   $e$ connects an
undercrossing to an overcrossing, then the band is a narrow
neighborhood of $e$ in $\Sigma$ meeting the square neighborhoods  of
the endpoints of $e$ along their sides in the obvious way. If both
endpoints of $e$ are
 undercrossings (resp.\   overcrossings), then one   takes the same
 band and half-twists it in the middle. The union of these squares
 and bands is a surface~$M$ containing~$\Gamma$ as a deformation
 retract. It is easy to check that~$\partial M$ is formed by
 disjoint
 copies of the 1-manifolds
 $K_{s_+}$ and  $K_{{s_-}}$ together with two   arcs joining the
 endpoints of $K_{s_+}$ and $K_{{s_-}}$. (These arcs come up as the sides of the squares obtained by thickening the endpoints of $K$.)  Therefore $\vert {s_+}\vert +\vert {{s_-}}\vert
 \leq
 b_0(\partial M)+ 1$, where $b_i$ denotes the $i$-th
 Betti number with coefficients in $\ZZ/2\ZZ$. Using the homology
 exact sequence of $(M, \partial M)$, the Poincar\'e duality, the connectedness of $M$,  and
 the Euler characteristic, we obtain
 $$b_0(\partial M)\leq b_0(M) +b_1(M,\partial M)=b_0(M)
 +b_1(M)=2-\chi(M)=n+1.$$
 Thus $\vert {s_+}\vert +\vert {{s_-}}\vert \leq n+2$. Together with
 \eqref{cc} this implies \eqref{cc-}.
\end{proof}

Theorem \ref{thhspann}  implies that for any knotoid $k$ in
$\Sigma$, \begin{equation}\label{cc27}\spn (k)\leq 4
 \cro(k), \end{equation} where $\cro (k)$ is the crossing number  of  $k $   defined as the
 minimal number of crossings in a  diagram of $k$.

\subsection{The case $\Sigma=S^2$}\label{section5.3} The
normalized bracket polynomial of knotoids in $S^2$ generalizes the
Jones
 polynomial of knots in $S^3$: for any knot $\kappa\subset S^3$, the
  polynomial    $ \langle \kappa^\bullet \rangle_\circ    $ is obtained from
 the Jones polynomial of $ \kappa  $ (belonging to $ \ZZ[  t^{\pm 1/4}]$) by the
 substitution $t^{\pm 1/4}=A^{\mp 1}$. For knotoids in $S^2$, Formula \eqref{cc27} has the following
 addendum.

  \begin{theor}\label{thspannew}  For a  knotoid $k$ in
  $S^2$,   we have $\spn (k)= 4
 \cro(k) $ if and only if $k=\kappa^\bullet$, where $\kappa $ is an alternating knot  in
 $S^3$. In particular,    $\spn (k)\leq  4
 \cro(k) -2$ for any pure knotoid  $k$ in
  $S^2$.
  \end{theor}

\begin{proof} If  $k=\kappa^\bullet$ for    an alternating knot $\kappa $,
then we  can present~$\kappa$ by a
  reduced alternating knot diagram $D$. Removing  from $D$ a small open arc disjoint from the
  crossings, we obtain  a   knotoid diagram~$K $ of~$k$ such that $\langle K \rangle=\langle D \rangle$. Then    $$\spn (k) =\spn
(\langle K \rangle)=\spn (\langle D \rangle)= 4
 \cro(D) = 4
 \cro(K) \geq
 4\cro(k), $$
 where the third equality is a well known property of reduced alternating knot
 diagrams, see \cite{ka}.    Combining with   \eqref{cc27}, we obtain   $\spn (k)=
 4\cro(k)$.

 To prove the converse, we need more terminology.    A knotoid diagram is {\it alternating} if traversing the diagram
from the leg to the tail one meets
   under- and over-crossings  in an alternating order. A simple geometric argument shows that all alternating knotoid diagrams in $S^2$ have complexity 0.
   (For a diagram $K$ of   positive complexity consider the region of $S^2-K$ adjacent to the head of $K$. This region is not adjacent to the leg of~$K$.
   Analyzing the over/under-passes of the edges of this region, one easily observes that $K$ cannot be alternating.)

   Recall
   that for any  knotoid diagrams $K_1, K_2$ in $S^2$, we can form
    a product knotoid diagram  $K_1K_2 \subset S^2$   (see Section \ref{kd3}).  We call  a knotoid diagram   $K\subset S^2$ {\it prime} if

(i) every embedded circle in~$S^2$ meeting $K$ transversely in one
point bounds a regular neighborhood of one of the endpoints of $K$
and

(ii) every  embedded circle in~$S^2$ meeting $K$ transversely in two
points bounds a disk in~$S^2$ meeting $K$ along a proper embedded
arc or along two disjoint   embedded arcs adjacent to the endpoints
of $K$.

Condition (i) means that $K$ is not  a product of two non-trivial
knotoid diagrams.   An induction on the number of crossings shows
that every knotoid diagram splits as a product of a finite number of
knotoid diagrams satisfying (i). These diagrams may not   satisfy
(ii). If a diagram $K$ of a knotoid~$k$ does not satisfy (ii), then
$K$ can be obtained from some other knotoid diagram by tying a
non-trivial knot in a small neighborhood of a generic point. Pushing
this knot towards the head of $K$, we obtain a   knotoid diagram
 of $k$ that has the same number of crossings as $K$ and splits a product of two non-trivial knotoid diagrams. An induction on the
number of crossings shows that for any knotoid diagram $K$ of a
knotoid $k$, there is a knotoid diagram $K'$ of $k$ such that $\cro
(K')=\cro (K)$ and $K'$ splits as a product of prime knotoid
diagrams.

We claim that   any prime knotoid diagram $K\subset S^2$ satisfying
$\spn (K)= 4
 \cro(K) $  is      alternating. The   argument is parallel to the one in
\cite{tu1} and proceeds as follows. We use the notation
 introduced in the proof of Theorem \ref{thhspann}. The formula $\spn (K)= 4
 \cro(K) $ implies that $\vert {s_+}\vert +\vert {{s_-}}\vert = n+2$.  Hence,
 $b_0(\partial M)= b_0(M) +b_1(M,\partial M)$. The latter equality holds if and only if the
 inclusion homomorphism $H_1(M ; \ZZ/2\ZZ)\to H_1(M, \partial
 M; \ZZ/2\ZZ)$ is equal to 0. This is possible if and only if
 the   intersection form
$ H_1(M; \ZZ/2\ZZ)\times H_1(M; \ZZ/2\ZZ)\to \ZZ/2\ZZ$ is zero.
Since $K$ is prime, for any
 edge $e$
 of $\Gamma$ connecting two 4-valent vertices,
 the  regions of $S^2-\Gamma$  adjacent to $e$ are distinct and their
 closures
have no common edges besides~$e$. The boundaries of these closures
are cycles in $\Gamma\subset M$. If $e$ connects two  undercrossings
 or two
  overcrossings, then the intersection number in $M$ of these two cycles   is equal to $1(\modu 2)$ which contradicts the triviality of the
  intersection form. Hence $\Gamma$ has no such edges and $K$ is alternating.

We can now accomplish the proof of the theorem.   Let~$k$ be a
knotoid in~$S^2$ such that $\spn (k)=
 4\cro(k)$. Then any minimal diagram $K$ of $k$ satisfies  $\spn (K)=
 4\cro(K)$.  By the argument above, we can choose $K$ so that it is a product of prime diagrams $K_1, \ldots, K_r$.
  Observe that both  numbers $\spn (K)$ and $
 \cro(K) $ are additive with respect to multiplication of knotoid diagrams. The assumption $\spn (K)=
 4\cro(K)$   and the inequality \eqref{cc27} imply that  $\spn (K_i)=
 4\cro(K_i)$ for $i=1, \ldots, r$.  By the previous paragraph,  each $K_i$  is an alternating knotoid diagrams (of
  complexity 0).
 Therefore there are   alternating knots $\kappa_1,..., \kappa_r\subset S^3$ such that $k= \kappa^\bullet$ for  $\kappa=\kappa_1 + \cdots +
\kappa_r$. It remains to observe that the knot $\kappa$  is
alternating. \end{proof}

\subsection{Example}\label{kd6===--}   For the pure knotoid $\varphi$ in $S^2$ shown in Figure~\ref{fig4}, we have $\langle \varphi\rangle_\circ=A^{4}+A^{6}-A^{10}$.
 Clearly, $\spn (\varphi)=6$ and $\cro (\varphi)=2$. In this case,
 the inequality $\spn (\kappa)\leq 4\cro (\kappa)-2$ is an equality.

   \subsection{Remarks}\label{kd6===56--} 1.  Kauffman's notions of a virtual knot diagram and a virtual knot
extend to knotoids in the obvious way. The theory of virtual
knotoids is equivalent to the theory of knotoids in closed connected
oriented
 surfaces considered up to orientation-preserving homeomorphisms and attaching handles in the complement of knotoid diagrams.

 2. The following observation is due to Oleg Viro. Every knotoid~$k$ (or virtual knotoid) in an oriented surface  determines an oriented virtual knot
 through   the \lq\lq virtual closure": the
 endpoints of~$k$ are connected by a simple arc in the ambient surface;  all
 intersections of the arc  with~$k$ are declared to be virtual. This construction allows one to
 apply to knotoids the   invariants of virtual knots. For
 example, the normalized bracket polynomial of knotoids introduced above results in
 this way from the normalized bracket polynomial of virtual  knots.
 Using the virtual closure, we can introduce the Khovanov homology and the  Khovanov-Rozansky
 homology of knotoids (and more generally of multi-knotoids).

3. Any knotoid $k$   in an oriented surface $\Sigma$ determines an
oriented knot~$k^\circ$ in the 3-manifold $\Sigma'\times [0,1]$,
where $\Sigma'=\Sigma\# (S^1\times S^1)$. To obtain~$k^\circ$,
remove the interiors of  disjoint regular neighborhoods $B_0,B_1
\subset \Sigma$ of the endpoints of~$k$ and glue $\partial  B_0$ to
$\partial B_1$ along an orientation-reversing homeomorphism carrying
the point $k\cap \partial B_0$ to the point $k\cap \partial B_1$.
Then~$k^\circ$ is the image of   $k\cap (\Sigma \setminus \Int
(B_0\cup B_1))$ under this gluing. A similar construction applies to
multi-knotoids, where the genus of the ambient surface increases by
 the number of interval components. In particular, any knotoid
 in~$S^2$ determines an oriented knot in $ S^1\times S^1 \times [0,1]$.

 4.     The notion of a finite type invariant of knots directly extends
to knotoids. It would be interesting to extend to knotoids other
  knot invariants:  the Kontsevich integral, the colored
Jones polynomials,      the Heegaard-Floer homology, etc.

5. For any knot $\kappa \subset  S^3$, we have   $\cro
 (\kappa^\bullet)\leq \cro (\kappa)$. Conjecturally, $\cro
 (\kappa^\bullet)= \cro (\kappa)$. This would follow from the
 stronger conjecture that any minimal diagram of the knotoid $\kappa^\bullet$
 has complexity 0.

\section{Extended bracket polynomial of knotoids}\label{section67}

 \subsection{Polynomial $\langle\langle \, \rangle\rangle_\circ
 $}\label{two-variablepoly} We introduce   a
  2-variable extension of the   bracket polynomial of knotoids. Let $K$ be a knotoid
  diagram in $S^2$. Pick a shortcut $a\subset S^2$ for
  $K$
  (cf.\
  Section \ref{section4.1}).  Given a state~$s\in S(K) $, consider the smoothed 1-manifold $K_s\subset S^2$ and its segment component
  $k_s$. (It is understood that the smoothing of~$K$ is effected in  small neighboroods of the crossings disjoint
  from~$a$.)
  Note that~$k_s$ coincides with~$K$ in a small neighborhood of the endpoints of
  $K$. In particular, the   set   $\partial k_s=\partial a $   consists of the endpoints of
  $K$. We
  orient $K$,
  $k_s$, and $a$ from the leg of $K$ to the head of $K$. Let $k_s\cdot a$ be the algebraic number of intersections of $k_s$ with $a$, that is the number of times
  $k_s$ crosses $a$ from the right to the left
  minus the number of times $k_s$ crosses $a$ from the   left to the
  right (the   endpoints of $k_s$ and $a$ are not counted).
  Similarly, let $K\cdot a$ be the algebraic number of intersections of $K$ with
  $a$. We define a 2-variable Laurent polynomial $\langle\langle K \rangle\rangle_\circ \in \ZZ[A^{\pm 1}, u^{\pm
  1}]$ by
$$\langle\langle K \rangle\rangle_\circ
= (-A^3)^{- w(K)}\, u^{  - K\cdot a}\,  \sum_{s\in S(K)}
A^{\sigma_s}   u^{k_s\cdot a  } (-A^2-A^{-2})^{\vert s\vert -1} .$$
The definition  of $\langle\langle K
  \rangle\rangle_\circ$
  extends  word for word to multi-knotoid diagrams  in~$ S^2$, see
  Section~\ref{kd1-----}.
The following lemma shows
 that  the polynomial $\langle\langle K \rangle\rangle_\circ$ yields an
 invariant of   knotoids and multi-knotoids. This invariant   is
 denoted $\langle\langle \, \rangle\rangle_\circ
 $.

\begin{lemma}\label{le356}  The   polynomial $\langle\langle K \rangle\rangle_\circ$ does not depend on the choice of the shortcut~$a$ and is invariant
 under the   Reidemeister moves on~$K$. \end{lemma}

\begin{proof} As we know, any two shortcuts for $K$ are isotopic in the class of embedded arcs in $S^2$ connecting the endpoints of $K$.
Therefore, to verify the independence of $a$, it is enough to
analyze the following  three  local transformations of~$a$:

(1) pulling   $a$ across a strand of $K$ (this adds two   points to
$a\cap K$);

(2)   pulling   $a$ across a double point of $ K$;

(3) adding a   curl to $a$ near an endpoint of $K$ (this adds a
  point to  $a\cap K$).

  The transformations (1) and (2) preserve the numbers $K\cdot a$ and $   k_s\cdot
  a$ for all states $s$ of $K$. The transformation (3) preserves    $   k_s\cdot
  a - K\cdot a$ for all   $s$. Hence,  $\langle\langle K \rangle\rangle_\circ$ is preserved under these transformations and does not depend on~$a$.

Consider  the   \lq\lq unnormalized"  version $\langle\langle K, a
 \rangle\rangle $ of $\langle\langle K
 \rangle\rangle_\circ$ obtained by deleting   the factor $(-A^3)^{- w(K)}\, u^{  - K\cdot a}$.  The polynomial $\langle\langle
 K, a
 \rangle\rangle $ depends on $a$ (hence the notation) but does not depend on the orientation
 of
 $K$ (to compute $k_s\cdot a$ one needs only to remember which endpoint is the leg and which one is the head). The polynomial $\langle\langle
 K, a
 \rangle\rangle $
 satisfies Kauffman's recursive relation   \begin{equation}\label{kau} \langle\langle
 K, a
 \rangle\rangle =A \langle\langle K_A, a
 \rangle\rangle +A^{-1} \langle\langle K_B, a
 \rangle\rangle ,\end{equation}
 where   $K_A$ is obtained from $K$ by the A-smoothing at
 a certain crossing and~$K_B$ is obtained from $K$ by the B-smoothing at
 the same crossing. Here the diagrams $K$, $K_A$, $K_B$   are unoriented and share the same leg and   head. (At least one of these
  diagrams  has a circle component so  that Formula~\eqref{kau} necessarily involves
    multi-knotoids.)
 The standard argument based on \eqref{kau} shows that   $\langle\langle
 K, a
 \rangle\rangle $ is invariant  under the second and third Reidemeister moves on~$K$
 and is multiplied by $(-A^3)^{\pm 1}$ under the first Reidemeister
 moves provided these moves proceed away from $a$. Such moves also preserve the number $K\cdot a$ and therefore they preserve  $\langle\langle K \rangle\rangle_\circ$.
 Since the
polynomial  $\langle\langle K \rangle\rangle_\circ$  does not depend
on~$a$, it is invariant under all Reidemeister moves on $K$.
\end{proof}

 \subsection{Special values}\label{spevayspans} For any knotoid $k$
 in $S^2$,
 $$\langle\langle k \rangle\rangle_\circ (A, u=1)=\langle k
 \rangle_\circ  ,$$
 $$\langle\langle k \rangle\rangle_\circ (A, u=-A^3)=\langle k_-
 \rangle_\circ  \quad {\text {and}} \quad \langle\langle k \rangle\rangle_\circ (A, u=-A^{-3})=\langle k_+
 \rangle_\circ  .$$
 These formulas show that the polynomial $\langle\langle k \rangle\rangle_\circ$ interpolates between the normalized bracket polynomials of $k$, $ k_-$,
  and $k_+$. The first formula is obvious and the other two are obtained by
 applying \eqref{kau} to all crossings of $K$ viewed as crossings of   $K\cup a$.
  This reduces the computation of $\langle\langle k \rangle\rangle_\circ (A,
 -A^{\pm 3})$ to the computation of the  bracket polynomial of
 the diagram of an unknot formed by the arcs  $k_s$ and $a$, where $k_s$ passes everywhere over (resp.\ under) $a$.
 The latter polynomial   is equal to $(-A^3)^{\pm {k_s\cdot a}}$.

 For example,
 $\langle\langle \varphi
\rangle\rangle_\circ= A^4+ (A^6-A^{10})u^2 $.  The substitutions
$u=1$, $u=-A^3$, and $u=-A^{-3}$ produce the normalized bracket
polynomial  of $\varphi$, of the  left-handed  trefoil, and of the
unknot, respectively.

Note finally that if a knotoid $k$ is a knot, then $\langle\langle k
\rangle\rangle_\circ = \langle k  \rangle_\circ \in \ZZ[A^{\pm 1}]$.

 \subsection{The $A$-span and the
 $u$-span}\label{two-variablepolyspans} For a polynomial $F \in \ZZ[A^{\pm 1}, u^{\pm
  1}]$, we   define two  numbers $\spn_A(F)$ and   $\spn_u(F)$. Let us expand $F $ as a finite sum $\sum_{i, j\in \ZZ} F_{i,j} A^i u^j$, where $F_{i,j}\in \ZZ$. If $F\neq 0$, then
 $\spn_A(F) =i_+-i_-$, where $i_+$ (resp.\ $i_-$) is the maximal (resp.\ the
minimal) integer~$i$ such that $F_{i,j}\neq 0$ for some $j$.
Similarly, $\spn_u(F) =j_+-j_-$, where $j_+$ (resp.\ $j_-$) is the
maximal (resp.\ the minimal) integer~$j$ such that $F_{i,j}\neq 0$
for some $i$. By definition,  $\spn_A(0) =\spn_u(0) =-\infty$.

 For
a knotoid $k$ in $S^2$, set $\spn_A(k)=\spn_A(\langle\langle k
\rangle\rangle_\circ)$ and $\spn_u(k)=\spn_u(\langle\langle k
\rangle\rangle_\circ)$. Both these numbers are even (non-negative)
integers. Clearly, \begin{equation}\label{inn} \spn(k)\leq \spn_A(k)
\leq 4 \cro (k)\quad {\text {and}}\quad \spn_u(k) \leq 2 c
(k),\end{equation} where the first two inequalities are obvious and
the third inequality is proven similarly to \eqref{cc27}. For
example, $\spn_A(\varphi)=\spn (\varphi)=6$ and $\spn_u(\varphi)=2$.
Here two of the inequalities \eqref{inn}  are equalities.

 \subsection{The skein relation}\label{spevayspans+} The polynomial $\langle\langle
 \,
 \rangle\rangle_\circ$ satisfies the skein relation
\begin{equation}\label{ske} -A^4 \langle\langle K_+ \rangle\rangle_\circ
+A^{-4}\langle\langle K_- \rangle\rangle_\circ
 =(A^2-A^{-2}) \langle\langle K_0 \rangle\rangle_\circ \end{equation} similar to the
  skein relation for the Jones polynomial. Here $ K_+$, $
 K_-$, and $
 K_0$ are any  multi-knotoid diagrams in
 $S^2$  which are the same except in a small disk where they look
 like a positive crossing, a negative crossing, and a pair of disjoint
 embedded arcs, respectively, see Figure~\ref{fig7}. (We call such a triple $(K_+, K_-, K_0)$ a   {\it Conway
 triple}.) The proof of \eqref{ske} is the same
 as for knots, see \cite{ka}, \cite{li}.

 \begin{figure}[h,t]
\includegraphics[width=8cm,height=2.5cm]{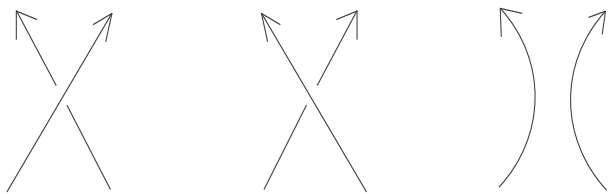}
\caption{A Conway triple in a disk}\label{fig7}
\end{figure}

\section{The skein algebra of knotoids}\label{section67ppp}

 \subsection{The algebra    $ \BB$} In analogy with skein  algebras of knots, we  define   a
skein algebra  of knotoids in~$S^2$.
  Let $\GK$ be the set of isotopy classes of multi-knotoids in $S^2$. Consider the Laurent polynomial ring $\Lambda=\ZZ[q^{\pm 1}, z^{\pm
 1}]$ and the free $\Lambda$-module $\Lambda [\GK]$ with basis $\GK$. Let $\BB$ be
the quotient of  $\Lambda [\GK]$    by
 the submodule generated by all vectors   $qK_+-q^{-1} K_- -z
 K_0$, where $(K_+, K_-, K_0)$ runs over the Conway triples of
 multi-knotoids. The obvious multiplication of  multi-knotoids (generalizing multiplication of knotoids)  turns $\BB$ into a $\Lambda$-algebra. The algebra $\BB$ has a unit represented by  the trivial knotoid.
 We will compute this algebra. In particular, we will show that $\BB$ is a commutative  polynomial $\Lambda$-algebra on a countable set of generators.

 To formulate our results, recall
  the definition of the
skein module of an oriented 3-manifold $M$ (see \cite{tu2},
\cite{pr}).
  Let $\Lu$ be the set of isotopy classes of  oriented  links in
 $M$ including the empty link~$\emptyset$. Three  oriented links $\ell_+, \ell_-, \ell_0\subset M$ form a {\it Conway
 triple} if they are   identical outside a
 ball in $ M$ while inside this ball they are as in Figure~\ref{fig7}.  Additionally, the triple $(\emptyset, \emptyset, {\text {a trivial knot}})$ is declared to be a Conway triple.
  The skein module $\Su(M)$ of $M$ is the quotient of the free $\Lambda$-module $\Lambda [\Lu]$ with basis
 $\Lu$  by
 the submodule generated by all vectors   $q\ell_+-q^{-1} \ell_- -z
 \ell_0$, where $(\ell_+, \ell_-, \ell_0)$ runs over the Conway triples in~$M$.

 For  an oriented surface $\Sigma$, the links in $\Sigma\times \RR$ can be represented by link diagrams in~$\Sigma$ in the usual way.
The skein module $\Su(\Sigma\times \RR)$ is a $\Lambda$-algebra with multiplication defined by placing
a diagram of the first link over a diagram of the second  link. The empty link  is the unit of this algebra.

 For  the annulus $A=S^1\times I$, where $I=[0,1]$, the $\Lambda$-algebra $\Au =\Su(A\times \RR) $ was fully computed in~ \cite{tu2}. We briefly recall the relevant results.
 Observe   that $\mathcal A=\oplus_{r\in \ZZ}\, \mathcal A_r$, where
$\mathcal A_r$ is the submodule generated by the  links homological to $r[S^1]$ in  $H_1(A)=\ZZ$. Here  $[S^1]\in H_1(A) $ is the generator determined by the counterclockwise orientation of $S^1$.
Pick a point $p\in S^1$ and for each $r\in \ZZ$, consider an oriented knot diagram in $A$ formed by
 the segment $\{p\}\times I$ and an embedded arc  $\gamma_r\subset A$ leading from $ (p,1) $ to $(p,0)$ and   passing everywhere over $ \{p\}\times I$ (except at the endpoints). The choice of $\gamma_r$ is uniquely
 (up to isotopy in $A$) determined by the condition that the resulting diagram is homological to $r[S^1]$ in  $H_1(A)$. This diagram represents
a vector  $z_r\in \mathcal A_r$. By \cite{tu2}, $\Au$ is a  commutative  polynomial $\Lambda$-algebra on the
 generators $\{z_r\}_{r \neq 0}$.   Note that $z_0=(q-q^{-1}) z^{-1}\in \Lambda\subset \Au_0$ and that  the group of  orientation-preserving  self-homeomorphisms of
$A$ (generated by the Dehn twist about    $S^1\times \{1/2\}$) acts trivially on~$\Au$. The algebra~$\Au$   has been further studied by H.~Morton and his co-authors, see for instance~\cite{mor}.

\begin{theor}  \label{polyP}  The $\Lambda$-algebras $\Au$ and $ \BB$ are isomorphic.
  \end{theor}

  \begin{proof}  We    call   multi-knotoid diagrams in
$ S^2=\RR^2\cup \{\infty\}$  with leg   $0$ and  head~$\infty$ {\it
special}. Any multi-knotoid diagram in $S^2$ is isotopic to a
special one. If two special diagrams are isotopic in $S^2$, then
they are isotopic in the class of special diagrams. Therefore, to
compute $\BB$ it is enough to  use only  special   diagrams.

We can  cut from any special  multi-knotoid diagram   in $S^2$ small
open regular neighborhoods of the endpoints. The remaining part of
$S^2$  can be identified with   $A=S^1\times I$.
  This
  allows us to switch from  the language of  special multi-knotoid diagrams in $S^2$  to the language of  multi-knotoid diagrams  in
      $  A$ whose legs and   heads   lie respectively on the boundary circles $S^1\times \{0\}$ and $S^1\times \{1\}$.
      The latter diagrams are considered up to the Reidemeister moves and isotopy in    $A$.
      Note that the isotopy may move the legs and the heads on $\partial A$;
      as a consequence  there is no well-defined rotation number  (or winding number) of a diagram.

Every oriented link diagram $L$ in $ A$ determines (possibly after
slight deformation) a multi-knotoid diagram $L_p=L\cup (\{p\}\times
I)$ in~$A$, where~$\{p\}\times I$ passes everywhere over $L$. The
Reidemeister moves and isotopies on $L$ are translated into the
Reidemeister moves and isotopies on $L_p$. Therefore the formula
$L\mapsto L_p$ defines a map from the set   of isotopy classes of
oriented links in $ A\times \RR$ to the set~$\GK$ of multi-knotoids
in~$S^2$. This map carries Conway triples of links to Conway triples
of multi-knotoids and induces   a $\Lambda$-homomorphism
$\psi:\Au\to \BB$.

We claim that  $\psi $ is an isomorphism. We first establish the
surjectivity. Let us call a   multi-knotoid   diagram $K $   {\it
ascending} if the segment component $C_K$ of~$K$ lies everywhere
over the other components, and moving along $C_K$ from the leg to
the head we first encounter every self-crossing   of $C_K$ as  an
underpass. Any  ascending diagram~$K$ in $A$ can be transformed by
the Reidemeister moves and isotopy of~$C_K$ into a multi-knotoid
diagram of type~$L_p$ as above. Hence, the generators of $\BB$
represented by the ascending diagrams lie in  $\psi(\Au)$.
  Given a  non-ascending
multi-knotoid diagram~$K\subset A$ with $m$ crossings,    we can
change its overcrossings to undercrossings in a unique way to obtain
an ascending diagram~$K'$. Changing   one crossing  at a time and
using the skein relation, we can recursively expand  $K$   as a
linear combination of     $K'$ and     diagrams with $<m$ crossings.
This shows by induction on $m$ that the generator of $\BB$
represented by $K$ lies in~$\psi(\Au)$. Hence, $\psi$ is surjective.

One may use a similar method to prove the injectivity of $\psi$. The
idea is to  define a map $\BB\to \Au$  by   $K\mapsto K-C_K$ on the
ascending diagrams  and then   extend this   to arbitrary
multi-knotoid diagrams using the recursive expansion above.  The
difficult part is to show that this gives a well defined map $\BB\to
\Au$. Then it is easy  to show that this map    is   inverse to
$\psi$. This approach is similar to the   Lickorish-Millett
construction of HOMFLYPT, see~\cite{li}.

We give  another proof of the injectivity of $\psi$.   We   define
for any integer $N$ a homomorphism $\mu_N:\BB\to \Au$ as follows.
Given a multi-knotoid diagram~$K$ in~$A$, we connect the endpoints
of $K$ by an embedded arc $\gamma=\gamma_{K,N} \subset A$ such that
$K\cup \gamma $ is homological to $N[S^1]\in H_1(A)$. Here the
orientation of $K\cup \gamma $ extends the one of $K$. Note   that
such arc $\gamma $ always exists and is unique up to isotopy
constant on $\partial \gamma$. We turn $K\cup \gamma $ into a
diagram of an oriented link by declaring that $\gamma $ passes
everywhere over $K$ (except at the endpoints). The isotopy class of
this link is preserved under the Reidemeister moves and isotopy of
$K$ in $A$. Moreover, the transformation $K\mapsto K\cup
\gamma_{K,N}$ carries Conway triples of multi-knotoid diagrams in
$A$ to Conway triples of links in $A\times \RR$. Therefore this
transformation defines a $\Lambda$-homomorphism $\mu_N:\BB\to \Au$.
 It follows from the  definitions that
\begin{equation}\label{rops}\mu_N \psi(a)=z_{N-r}\,  a \end{equation}
for any $r \in \ZZ$ and any $a\in \mathcal A_r$.

Every vector   $a\in \Ker \psi$  expands as $a=\sum_{r\in \ZZ} \, a_r$, where $a_r\in \mathcal A_r$ for all $r$.
Formula~\eqref{rops}   implies that
$\sum_{r} \,  z_{N-r}\, a_r=0$ for all $N\in \ZZ$. Recall that each $a_r$ is a polynomial in the generators
$\{z_s\}_{s\neq 0}$. For any $r_0\in \ZZ$, we can take  $N$  big enough so that  the generator  $z_{N-r_0}$ appears in  the sum $\sum_{r} \,  z_{N-r}\, a_r$  only as the factor in the term
$z_{N-r_0}\, a_{r_0}$. Since this sum is equal to zero,   $a_{r_0}=0$. Thus,   $a=0$ and $\psi$ is an isomorphism.
  \end{proof}

 \subsection{Remarks}\label{khjyu--} 1. Composing the projection $\GK\to \BB$ with   $\psi^{-1}$, we obtain a map $\mathcal P:\GK\to \Au$. This map yields
an   invariant  of multi-knotoids in $S^2$  extending   the HOMFLYPT
polynomial~$P$ of oriented links in $S^3$: if $\ell$ is an oriented
link in $S^3$ and $\ell^\bullet$ is a multi-knotoid in $S^2$
obtained by removing from a diagram of $\ell$ a small subarc
$\alpha$ (disjoint from the crossings), then $\mathcal P
(\ell^\bullet)=P(\ell)\in \Lambda\subset \Au$. Note that
$\ell^\bullet\in \GK$ may depend on the choice of the component of
$\ell$ containing $\alpha$ but depends neither on the choice of
$\alpha$ on this component nor on the choice of the diagram
of~$\ell$. Formula~\eqref{ske}  implies that the polynomial
$\langle\langle\, \rangle\rangle_\circ$  is determined by $\mathcal
P$.

2.    The results of this section can be reformulated in terms of
  theta-links, see Remark \ref{section3.5}.3.    One can define the skein
  relations
  for the theta-links as for links allowing the two strands   in the   relations to lie on
  the link components or on the 0-labeled edge of the theta-curve
  (but not on the $\pm$-labeled edges). The generalization of Theorem \ref{th777}   to multi-knotoids  mentioned in Remark \ref{section3.5}.3 implies that the skein algebra of multi-knotoids $\BB$ is
  isomorphic to the skein algebra of simple theta-links in~$S^3$.

3. One can similarly introduce the   algebras of multi-knotoids (or,
equivalently, of  simple theta-links)  modulo the bracket relation
\eqref{kau} or modulo the 4-term Kauffman skein relation used to
define the 2-variable Kauffman polynomial  of links. The resulting
algebras are
 isomorphic to the corresponding
  skein algebras of the annulus computed in~\cite{tu2}.

\section{Knotoids in $\RR^2$}\label{kd6==ppp}

 Since the knotoid diagrams in $S^2$ are usually drawn in $\RR^2$, it may be useful to  compare the sets $\mathcal K (\RR^2)$ and $\mathcal K (S^2)$.
  The inclusion $\RR^2 \hookrightarrow S^2$
allows us to view any knotoid diagram in $\RR^2$ as a knotoid
diagram in $S^2$ and induces thus an {\it inclusion map} $ \iota:
\mathcal K (\RR^2)\to \mathcal K (S^2)$.
  Given a knotoid in $S^2$, we can represent it by a
normal  diagram and consider the equivalence class of this diagram
in $\mathcal K (\RR^2)$. This defines a  map $\rho :\mathcal K
(S^2)\to \mathcal K (\RR^2)$. Clearly, $\iota \circ \rho=\id$ so
that $\iota$ is surjective.

As in Sections~\ref{section0} and~\ref{section1}, we have three
 basic involutions ${\rm rev}$, ${\rm sym}$, and ${\rm mir} $ on $\mathcal K (\RR^2)$. The maps $\iota$ and $\rho$   are
equivariant with respect to these involutions.

We now give examples of non-trivial knotoids in $\RR^2$  that are
trivial
  in $S^2$, i.e., are carried  by   $\iota $ to the trivial knotoid in~$S^2$.  Thus,   $\iota$ is not
  injective.

    Figure~\ref{fig1} represents a knotoid $U \in
\mathcal K (\RR^2)$ and its images under the basic involutions.
These knotoids are called {\it unifoils}.  Note  that $ ({\rm sym}
\circ {\rm mir} \circ {\rm rev}) (U)=U$.  Using isotopy and
$\Omega_1$, one easily observes that the unifolis
 are  trivial in~$S^2$.

Figure~\ref{fig2} represents two  knotoids $B_1,  B_2 \in \mathcal K
(\RR^2)$. These knotoids   and their images under the basic
involutions are called {\it bifoils}.
 As an exercise, the reader may     check that   ${\rm rev} (B_1)=B_1$,
${\rm rev}(B_2)= {\rm mir} (B_2)$, and  $ B_2 $  is trivial in
$S^2$.

We claim that   the unifoils and the bifoils are non-trivial
knotoids (in $\RR^2$). To prove this claim, we
 define for knotoids  in $\RR^2$  a $3$-variable polynomial
  $[\, ]_\circ $ with values in the ring $\ZZ[A^{\pm 1},
u^{\pm 1},v]$.   Given a state $s\in S(K)$ on a knotoid diagram
$K\subset \RR^2$, every circle component  of the 1-manifold $K_s $
bounds a disk in~$\RR^2$. This disk may either be disjoint from the
  segment component of $K_s$ or contain  this segment component.
Let $p_s$ (resp.\ $q_s$) be the number of circle components of~$K_s$
of the first (resp.\ the second) type. Clearly, $p_s+q_s=\vert
s\vert -1$. Set  $$[K]_\circ = (-A^3)^{- w(K)}\, u^{  - K\cdot a}\,
\sum_{s\in S(K)} A^{\sigma_s} u^{k_s\cdot a  } (-A^2-A^{-2})^{p_s}
v^{q_s} .$$   Standard computations show that  this   is an
invariant of knotoids in~$\RR^2$.  The polynomial $[K]_\circ $ is
invariant
 under the reversion of knotoids and   changes
 via
  $A\mapsto A^{-1}$ under mirror reflection and    symmetry
  in~$\RR^2$. For $v=-A^2-A^{-2}$, we recover the
  polynomial $\langle \langle\, \rangle \rangle_\circ$ from Section \ref{two-variablepoly}.

  Direct computations show that
  $[ U ]_\circ=-A^{4} -A^{2}v$,
 $$  [ B_1 ]_\circ= A^{4}+ 2A^6u^2 +A^{8}u^2v\quad {\text {and}} \quad
 [ B_2 ]_\circ= (A^{2} +A^{-2}+v) u^2 +1 .$$
 Therefore the  knotoids
 $U$, $B_1$, $B_2$ are non-trivial and mutually distinct.

Figure~\ref{fig4} represents a knotoid $\varphi \in \mathcal K
(\RR^2)$ and its images under the  involutions ${\rm sym}$  and
${\rm mir}$. It is easy to see that $\iota(\varphi)=\iota(B_1)$ and
therefore the knotoid $\iota(\varphi) \in  \mathcal K (S^2)$
(denoted $\varphi$ in the previous sections) is invariant under
reversion. As we know, the knotoid  $\iota(\varphi)$ is non-trivial.

                     \end{document}